\documentclass[reqno]{amsart}
\usepackage{amssymb,amsmath,hyperref}
\usepackage{amsrefs, float}
\usepackage[foot]{amsaddr}
\usepackage{bbold,stackrel, multicol}
\usepackage{mathrsfs}

\newcommand{\Z}{{\textsf{\textup{Z}}}}

\newtheorem{thm}{Theorem}

\newtheorem{cor}[thm]{Corollary}
\newtheorem{defi}[thm]{Definition}
\newtheorem{rem}[thm]{Remark}
\newtheorem{nota}[thm]{Notation}

\newtheorem{princ}[thm]{Principle}

\newtheorem{ack}[thm]{Acknowledgement}

\newtheorem*{tempo*}{Template}

\newcommand\be{\begin{equation}}
\newcommand\ee{\end{equation}} 

\usepackage{amsmath,amsfonts} 

\usepackage[applemac]{inputenc}

\def\blambda{\pmb{\lambda}}

\def\bdefi{\begin{defi}}
\def\edefi{\end{defi}}
\def\bnota{\begin{nota}\rm}
\def\enota{\end{nota}}
\def\FIVE{\Pi_{1}^{1}\text{-\textup{\textsf{CA}}}_{0}}
\def\SIX{\Pi_{2}^{1}\text{-\textsf{\textup{CA}}}_{0}}

\def\SIXK{\Pi_{k}^{1}\text{-\textsf{\textup{CA}}}_{0}^{\omega}}

\def\ATR{\textup{\textsf{ATR}}}

\def\ZF{\textup{\textsf{ZF}}}

\def\osc{\textup{\textsf{osc}}}

\def\L{\textsf{\textup{L}}}

 \def\r{\mathbb{r}}

\def\RCA{\textup{\textsf{RCA}}}
\def\({\textup{(}}
\def\){\textup{)}}

\def\RCAo{\textup{\textsf{RCA}}_{0}^{\omega}}
\def\ACAo{\textup{\textsf{ACA}}_{0}^{\omega}}

\def\WKL{\textup{\textsf{WKL}}}

\def\bye{\end{document}}
\def\N{{\mathbb  N}}
\def\Q{{\mathbb  Q}}
\def\R{{\mathbb  R}}

\def\SS{\textup{\textsf{S}}}

\def\di{\rightarrow}

\def\asa{\leftrightarrow}

\def\ACA{\textup{\textsf{ACA}}}

\def\QFAC{\textup{\textsf{QF-AC}}}

\def\PHP{\textup{\textsf{PHP}}}

\def\osc{\textup{\textsf{osc}}}

\def\alt{\textup{\textsf{alt}}}

\def\enum{\textup{\textsf{enum}}}

\def\cocode{\textup{\textsf{cocode}}}

\def\NIN{\textup{\textsf{NIN}}}

\def\BCT{\textup{\textsf{BCT}}}

\def\IND{\textup{\textsf{IND}}}

\def\eps{\varepsilon}

\def\ECF{\textup{\textsf{ECF}}}

\usepackage{graphicx}
\usepackage{tikz}
\usetikzlibrary{matrix}
\usepackage{comment,tikz-cd}

\numberwithin{equation}{section}
\numberwithin{thm}{section}

\begin{document}
\title{Approximation Theorems Throughout Reverse Mathematics}
\author{Sam Sanders}
\address{Department of Philosophy II, RUB Bochum, Germany}
\email{sasander@me.com}
\begin{abstract}
Reverse Mathematics (RM for short) is a program in the foundations of mathematics where the aim is to find the minimal axioms
needed to prove a given theorem of ordinary mathematics.  Generally, the minimal axioms are {equivalent} to the theorem at hand, assuming a weak logical system called the {base theory}.  
Moreover, many  theorems are either provable in the base theory or equivalent to one of four logical systems, together called the \emph{Big Five}.  
For instance, the \emph{Weierstrass approximation theorem}, i.e.\ that a continuous function can be approximated uniformly by a sequence of polynomials, has been classified in RM as being equivalent to \emph{weak K\"onig's lemma}, the second Big Five system.  
In this paper, we study approximation theorems for \emph{discontinuous} functions via Bernstein polynomials from the literature.   
We obtain many equivalences between the latter and weak K\"onig's lemma.
We also show that {slight} variations of these approximation theorems fall {far} outside of the Big Five but fit in the recently developed RM of new `big' systems, namely the uncountability of $\R$, 
the enumeration principle for countable sets, the pigeon-hole principle for measure, and the Baire category theorem. 
\end{abstract}

\setcounter{page}{0}
\tableofcontents
\thispagestyle{empty}
\newpage

\maketitle
\thispagestyle{empty}

\vspace{-2mm}
\section{Introduction and preliminares}\label{intro}

\subsection{Aim and motivation}\label{sintro}
The aim of the program \emph{Reverse Mathematics} (RM for short; see Section \ref{prelim} for an introduction) is to find the minimal axioms needed to prove a given theorem of ordinary mathematics.
Generally, the minimal axioms are {equivalent} to the theorem at hand, assuming a weak logical system called the {base theory}.  
The \emph{Big Five phenomenon} is a central topic in RM, as follows. 
\begin{quote}
[...] we would still claim that the great majority of the theorems from classical mathematics are equivalent to one of the big five. This phenomenon is still quite striking. Though we have some sense of why this phenomenon occurs, we really do not have a clear explanation for it, let alone a strictly logical or mathematical reason for it. The way I view it, gaining a greater understanding of this phenomenon is currently one of the driving questions behind reverse mathematics. (see \cite{montahue}*{p.\ 432})
\end{quote}
A natural example is the equivalence between the Weierstrass approximation theorem and weak K\"onig's lemma from \cite{simpson2}*{IV.2.5}.
In \cite{dagsamXIV}, Dag Normann and the author greatly {extend} the Big Five phenomenon by establishing numerous equivalences involving the \textbf{second-order} Big Five systems on one hand, 
and well-known \textbf{third-order} theorems from analysis about possibly discontinuous functions on the other hand, working in Kohlenbach's \emph{higher-order} RM (see Section \ref{prelim}).  
Moreover, following \cite{dagsamXIV}*{\S2.8}, {slight} variations/generalisations of these third-order theorems cannot be proved from the Big Five and {much} stronger systems.  
Nonetheless, an important message of \cite{dagsamXIV} is just how similar second- and higher-order RM can be, as the latter reinforces the existing Big Five with many further examples

\smallskip

In this paper, we develop the RM-study of approximation theorems, often involving \emph{Bernstein polynomials} $B_{n}$, defined as follows:
\be\label{laquel}
B_n(f,x):=\sum_{k=0}^{n} f\big({k}/{n}\big)p_{n,k}(x) \textup{ where } p_{n, k}(x):= \binom{n}{k} x^k (1-x)^{n-k}. 
\ee
Weierstrass (\cite{weiapprox}) and Borel (\cite{opborrelen}*{p.\ 80}) already study polynomial approximations while Bernstein (\cite{bern1}) provides the first explicit polynomials, namely $B_{n}$ as above, and showed that for continuous $f:[0,1]\di \R$, we have for all $x\in [0,1]$ that
\be\label{tink}
f(x)=\lim_{n\di \infty}B_{n}(f, x), 
\ee
where the convergence is uniform.
Approximation results for discontinuous functions are in \cites{overdeberg, klo, gaanwekatten2, piti} and\footnote{Picard studies approximations of Riemann integrable $f:[0,1]\di \R$ in \cite{piti}*{p.\ 252} \emph{related} to Berstein polynomials following \cite{klo}*{p.\ 64}.} take the form \eqref{tink} for points of continuity and the form \eqref{tink2} if the left and right limits $f(x-)$ and $f(x+)$ exist at $x\in (0,1)$: 
\be\label{tink2}\textstyle
\frac{f(x+)+f(x-)}{2}=\lim_{n\di \infty}B_{n}(f, x).
\ee
We stress that for general $f:[0,1]\di \R$, \eqref{tink} may hold at $x$ where $f$ is \emph{dis}continuous, i.e.\ \eqref{tink} is much weaker than continuity at $x$. 
For instance, Dirichlet's function $\mathbb{1_{Q}}$ satisfies \eqref{tink} for all rationals and is discontinuous everywhere. 

\smallskip

In this paper, we develop the RM-study of approximation theorems based on \eqref{tink} and \eqref{tink2} for {regulated}, semi-continuous, cliquish, and Riemann integrable functions. Scheeffer and Darboux study regulated functions in \cites{scheeffer, darb} without naming this class, while Baire introduced semi-continuity in \cite{beren2}.  Hankel studies cliquish functions using an equivalent definition in \cite{hankelwoot}.
Thus, there is plenty of historical motivation for our RM-study.  

\smallskip

In particular, we show that many approximation theorems for discontinuous functions are equivalent to weak K\"onig's lemma in Theorem \ref{LEMCOR}.  
We also show that {slight} variations fall {far} outside of the Big Five, but do yield equivalences for four new `Big' systems, namely \emph{the uncountability of $\R$} (\cite{samBIG}), the \emph{enumeration principle for countable sets} (\cite{dagsamXI}), the \emph{Baire category theorem} (\cite{samBIG2}), and the \emph{pigeon-hole principle} for measure (\cite{samBIG2}).  
These new Big systems boast equivalences involving principles based on pointwise continuity and it is perhaps surprising that the same holds for the weaker condition \eqref{tink}.  
Moreover, the equivalences for these new Big systems are \emph{robust} as follows:
\begin{quote}
A system is \emph{robust} if it is equivalent to small perturbations of itself. (\cite{montahue}*{p.\ 432}; emphasis in original)
\end{quote}
In particular, our equivalences generally go through for many variations of the function classes involved, with examples in Theorems \ref{tomma} and \ref{tach}. 
There is an apparent tension here with the earlier observation that slight variations in the theorems can greatly affect the logical strength required to prove them.  
We attempt to explain this phenomenon in Remark \ref{second-order-ish}.  

\smallskip

Finally, in Section~\ref{prelim}, we provide an introduction to RM; we introduce essential higher-order concepts in Sections \ref{lll}-\ref{couse}.
We establish the equivalences for weak K\"onig's lemma in Section \ref{worg} and the equivalences for the 
`new' big systems in Section \ref{bigger}-\ref{piho}.  

\subsection{Preliminaries and definitions}\label{helim}
We briefly introduce {Reverse Mathematics} in Section \ref{prelim}.
We introduce some essential axioms (Section \ref{lll}) and definitions (Section~\ref{cdef}) needed in the below.  
We discuss the definition of countable set in higher-order RM in Section \ref{couse}. 

\subsubsection{Reverse Mathematics}\label{prelim}
Reverse Mathematics (RM hereafter) is a program in the foundations of mathematics initiated around 1975 by Friedman (\cites{fried,fried2}) and developed extensively by Simpson and others (\cite{simpson2, simpson1,damurm}).  
The aim of RM is to identify the minimal axioms needed to prove theorems of ordinary, i.e.\ non-set theoretical, mathematics. 

\smallskip

First of all, we refer to \cite{stillebron} for a basic introduction to RM and to \cite{simpson2, simpson1,damurm} for an overview of RM.  We expect basic familiarity with RM, in particular Kohlenbach's \emph{higher-order} RM (\cite{kohlenbach2}) essential to this paper, including the base theory $\RCAo$, also introduced in Section \ref{rmbt}.   An extensive introduction can be found in e.g.\ \cites{dagsamIII, dagsamV, dagsamX} and elsewhere.  

\smallskip

Secondly, as to notations, we have chosen to include a brief introduction as a technical appendix, namely Section \ref{RMA}.  
All undefined notions may be found in the latter, while we do point out here that we shall sometimes use common notations from type theory.  For instance, the natural numbers are type $0$ objects, denoted $n^{0}$ or $n\in \N$.  
Similarly, elements of Baire space are type $1$ objects, denoted $f\in \N^{\N}$ or $f^{1}$.  Mappings from Baire space $\N^{\N}$ to $\N$ are denoted $Y:\N^{\N}\di \N$ or $Y^{2}$.   

\smallskip

Thirdly, the main topic of this paper is the RM-study of real analysis, for which the following notations suffice.  
Both in $\RCA_{0}$ and $\RCAo$, real numbers are given by Cauchy sequences (see Definition \ref{keepintireal} and \cite{simpson2}*{II.4.4}) and equality between reals `$=_{\R}$' has the same meaning in second- and higher-order RM.  
Functions on $\R$ are defined as mappings $\Phi$ from $\N^{\N}$ to $\N^{\N}$ that respect real equality, i.e.\ $x=_{\R}y\di \Phi(x)=_{\R}\Phi(y)$ for any $x, y\in \R$, which is also called \emph{function extensionality}.  

\smallskip

Fourth, experience bears out that the following fragment of the Axiom of Choice is often convenient when (first) proving equivalences.
This principle can often be omitted by developing a more sophisticated alternative proof (see e.g.\ \cite{dagsamIX}).  
\begin{princ}[$\QFAC^{0,1}$] For any $Y^{2}$, we have:
\be\label{bing}
(\forall n\in \N)(\exists f\in \N^{\N})(Y(f, n)=0 )\di (\exists (f_{n})_{n\in \N})(\forall n\in \N)(Y(f_{n}, n)=0 ).
\ee
\end{princ}
As discussed in \cite{kohlenbach2}*{Remark 3.13}, this principle is not provable in $\ZF$ while $\RCAo+\QFAC^{0,1}$ suffices 
to prove the local equivalence between (epsilon-delta) continuity and sequential continuity, which is also not provable in $\ZF$. 

\smallskip

Finally, the main difference between Friedman-Simpson and Kohlenbach's framework for RM is whether the language is restricted to \emph{second-order} objects or if one allows \emph{third-order} objects.   
An important message of \cite{dagsamXIV} and this paper is that the second-order Big Five are equivalent to third-order theorems concerning possibly discontinuous functions, as is also clear from Theorems \ref{LEMCOR} and \ref{flank}.

\subsubsection{Some comprehension functionals}\label{lll}
In second-order RM, the logical hardness of a theorem is measured via what fragment of the comprehension axiom (broadly construed) is needed for a proof.  
For this reason, we introduce some axioms and functionals related to \emph{higher-order comprehension} in this section.
We are mostly dealing with \emph{conventional} comprehension here, i.e.\ only parameters over $\N$ and $\N^{\N}$ are allowed in formula classes like $\Pi_{k}^{1}$ and $\Sigma_{k}^{1}$.  

\smallskip
\noindent
First of all, the functional $\varphi$ in $(\exists^{2})$ is also \emph{Kleene's quantifier $\exists^{2}$} and is clearly discontinuous at $f=11\dots$ in Cantor space:
\be\label{muk}\tag{$\exists^{2}$}
(\exists \varphi^{2}\leq_{2}1)(\forall f^{1})\big[(\exists n^{0})(f(n)=0) \asa \varphi(f)=0    \big]. 
\ee
In fact, $(\exists^{2})$ is equivalent to the existence of $F:\R\di\R$ such that $F(x)=1$ if $x>_{\R}0$, and $0$ otherwise (see \cite{kohlenbach2}*{Prop.\ 3.12}).  
Related to $(\exists^{2})$, the functional $\mu^{2}$ in $(\mu^{2})$ is called \emph{Feferman's $\mu$} (see \cite{avi2}) and may be found -with the same symbol- in Hilbert-Bernays' Grundlagen (\cite{hillebilly2}*{Supplement IV}):
\begin{align}\label{mu}\tag{$\mu^{2}$}
(\exists \mu^{2})(\forall f^{1})\big[ (\exists n)(f(n)=0) \di [f(\mu(f))=0&\wedge (\forall i<\mu(f))(f(i)\ne 0) ]\\
& \wedge [ (\forall n)(f(n)\ne0)\di   \mu(f)=0]    \big]. \notag
\end{align}
We have $(\exists^{2})\asa (\mu^{2})$ over $\RCAo$ (see \cite{kohlenbach2}*{\S3}) and $\ACAo\equiv\RCAo+(\exists^{2})$ proves the same sentences as $\ACA_{0}$ by \cite{hunterphd}*{Theorem~2.5}. 

\smallskip

Secondly, the functional $\SS^{2}$ in $(\SS^{2})$ is called \emph{the Suslin functional} (\cite{kohlenbach2}):
\be\tag{$\SS^{2}$}
(\exists\SS^{2}\leq_{2}1)(\forall f^{1})\big[  (\exists g^{1})(\forall n^{0})(f(\overline{g}n)=0)\asa \SS(f)=0  \big].
\ee
The system $\FIVE^{\omega}\equiv \RCAo+(\SS^{2})$ proves the same $\Pi_{3}^{1}$-sentences as $\FIVE$ by \cite{yamayamaharehare}*{Theorem 2.2}.   
By definition, the Suslin functional $\SS^{2}$ can decide whether a $\Sigma_{1}^{1}$-formula as in the left-hand side of $(\SS^{2})$ is true or false.   We similarly define the functional $\SS_{k}^{2}$ which decides the truth or falsity of $\Sigma_{k}^{1}$-formulas from $\L_{2}$; we also define 
the system $\SIXK$ as $\RCAo+(\SS_{k}^{2})$, where  $(\SS_{k}^{2})$ expresses that $\SS_{k}^{2}$ exists.  
We note that the operators $\nu_{n}$ from \cite{boekskeopendoen}*{p.\ 129} are essentially $\SS_{n}^{2}$ strengthened to return a witness (if existant) to the $\Sigma_{n}^{1}$-formula at hand.  

\smallskip

\noindent
Thirdly, full second-order arithmetic $\Z_{2}$ is readily derived from $\cup_{k}\SIXK$, or from:
\be\tag{$\exists^{3}$}
(\exists E^{3}\leq_{3}1)(\forall Y^{2})\big[  (\exists f^{1})(Y(f)=0)\asa E(Y)=0  \big], 
\ee
and we therefore define $\Z_{2}^{\Omega}\equiv \RCAo+(\exists^{3})$ and $\Z_{2}^\omega\equiv \cup_{k}\SIXK$, which are conservative over $\Z_{2}$ by \cite{hunterphd}*{Cor.\ 2.6}. 
Despite this close connection, $\Z_{2}^{\omega}$ and $\Z_{2}^{\Omega}$ can behave quite differently, as discussed in e.g.\ \cite{dagsamIII}*{\S2.2} and Section \ref{mintro}.   
The functional from $(\exists^{3})$ is also called `$\exists^{3}$', and we use the same convention for other functionals.  

\smallskip

Finally, Kleene's quantifier $\exists^{2}$ plays a crucial role throughout higher-order RM.  We recall that $(\exists^{2})$ is equivalent to the existence of a discontinuous function on $\R$ (or $\N^{\N}$) by \cite{kohlenbach2}*{Prop.\ 3.12}, using so-called Grilliot's trick.  Thus, $\neg(\exists^{2})$ is equivalent to Brouwer's theorem, i.e.\ all functions on the reals (and Baire space) are continuous.  We will often make use of the latter fact without explicitly pointing this out.

\subsubsection{Some definitions}\label{cdef}
We introduce some required definitions, stemming from mainstream mathematics.
We note that subsets of $\R$ are given by their characteristic functions as in Definition \ref{char}, well-known from measure and probability theory.
We shall generally work over $\ACAo$ as some definitions make little sense over $\RCAo$.

\smallskip

First of all, we make use the usual definition of (open) set, where $B(x, r)$ is the open ball with radius $r>0$ centred at $x\in \R$.
 We note that our notion of `measure zero' does not depend on (the existence of) the Lebesgue measure.
\bdefi[Sets]\label{char}~
\begin{itemize}
\item A subset $A\subset \R$ is given by its characteristic function $F_{A}:\R\di \{0,1\}$, i.e.\ we write $x\in A$ for $ F_{A}(x)=1$, for any $x\in \R$.
\item A subset $O\subset \R$ is \emph{open} in case $x\in O$ implies that there is $k\in \N$ such that $B(x, \frac{1}{2^{k}})\subset O$.
\item A subset $O\subset \R$ is \emph{RM-open} in case there are sequences of reals $(a_{n})_{n\in \N}, (b_{n})_{n\in \N}$ such that $O=\cup_{n\in \N}(a_{n}, b_{n})$.
\item A subset $C\subset \R$ is \emph{closed} if the complement $\R\setminus C$ is open. 
\item A subset $C\subset \R$ is \emph{RM-closed} if the complement $\R\setminus C$ is RM-open. 
\item A set $A\subset \R$ is \emph{enumerable} if there is a sequence of reals that includes all elements of $A$.
\item A set $A\subset \R$ is \emph{countable} if there is $Y: \R\di \N$ that is injective on $A$, i.e.\
\[
(\forall x, y\in A)( Y(x)=_{0}Y(y)\di x=_{\R}y).  
\]
\item A set $A\subset \R$ is \emph{measure zero} if for any $\eps>0$ there is a sequence of open intervals $(I_{n})_{n\in \N}$ such that $\cup_{n\in \N}I_{n}$ covers $A$ and $\eps>\sum_{n=0}^{\infty}|I_{n}|$. 
\item A set $A\subset \R$ is \emph{dense} in $B\subset \R$ if for $k\in \N,b\in B$, there is $a\in A$ with $|a-b|<\frac{1}{2^{k}}$.
\item A set $A\subset \R$ is \emph{nowhere dense} in $B\subset \R$ if $A$ is not dense in any open sub-interval of $B$.  
\end{itemize}
\edefi
\noindent

As discussed in Section \ref{couse}, the study of regulated functions already gives rise to open sets that 
do not come with additional representation beyond the second item in Definition~\ref{char}.  We will often assume $(\exists^{2})$ from Section \ref{lll} to guarantee that
basic objects like the unit interval are sets in the sense of Def.\ \ref{char}.  

\smallskip

Secondly, we study the following notions, many of which are well-known and hark back to the days of Baire, Darboux, Dini, Jordan, Hankel, and Volterra (\cites{jordel, beren,beren2,darb, volaarde2,hankelwoot,hankelijkheid,dinipi}).  
We use `sup' and other operators in the `virtual' or `comparative' way of second-order RM (see e.g.\ \cite{simpson2}*{X.1} or \cite{browner}).  In this way, a formula of the form `$\sup A>a$' or `$x\in \overline{S}$' makes sense as shorthand\footnote{For instance, `$\sup A> a$' simply abbreviates $(\exists x\in A)( x>a)$, while `$x\in \overline{S}$' means that there is a sequence of elements in $S$ converging to $x$.} for a formula in the language of all finite types, even when $\sup A$ or the closure $\overline{S}$ need not exist in $\RCAo$.  
\bdefi\label{flung} For $f:[0,1]\di \R$, we have the following definitions:
\begin{itemize}
\item $f$ is \emph{upper semi-continuous} at $x_{0}\in [0,1]$ if $f(x_{0})\geq_{\R}\lim\sup_{x\di x_{0}} f(x)$,
\item $f$ is \emph{lower semi-continuous} at $x_{0}\in [0,1]$ if $f(x_{0})\leq_{\R}\lim\inf_{x\di x_{0}} f(x)$,
\item $f$ \emph{has bounded variation} on $[0,1]$ if there is $k_{0}\in \N$ such that $k_{0}\geq \sum_{i=0}^{n} |f(x_{i})-f(x_{i+1})|$ 
for any partition $x_{0}=0 <x_{1}< \dots< x_{n-1}<x_{n}=1  $,
\item $f$ is \emph{regulated} if for every $x_{0}$ in the domain, the `left' and `right' limit $f(x_{0}-)=\lim_{x\di x_{0}-}f(x)$ and $f(x_{0}+)=\lim_{x\di x_{0}+}f(x)$ exist, 
\item $f$ is \emph{cadlag} if it is regulated and $f(x)=f(x+)$ for $x\in [0,1)$,
\item $f$ is a $U_{0}$-function \(\cite{jojo, lorentzg, gofer2, voordedorst}\) if it is regulated and for all $x\in (0,1)$: 
\be\label{zolk}
\min(f(x+), f(x-))\leq f(x)\leq \max(f(x+), f(x-)),
\ee
\item $f$ is \emph{Baire $1$} if it is the pointwise limit of a sequence of continuous functions,
\item $f$ is \emph{Baire 1$^{*}$} if\footnote{The notion of Baire 1$^{*}$ goes back to \cite{ellis} and equivalent definitions may be found in \cite{kerkje}.  
In particular, Baire 1$^{*}$ is equivalent to the Jayne-Rogers notion of \emph{piecewise continuity} from \cite{JR}.} there is a sequence of closed sets $(C_{n})_{n\in \N}$ such $[0,1]=\cup_{n\in \N}C_{n}$ and $f_{\upharpoonright C_{m}}$ is continuous for all $m\in \N$.
\item $f$ is \emph{quasi-continuous} at $x_{0}\in [0, 1]$ if for $ \epsilon > 0$ and any open neighbourhood $U$ of $x_{0}$, 
there is non-empty open ${ G\subset U}$ with $(\forall x\in G) (|f(x_{0})-f(x)|<\eps)$.
\item $f$ is \emph{cliquish} at $x_{0}\in [0, 1]$ if for $ \epsilon > 0$ and any open neighbourhood $U$ of $x_{0}$, 
there is a non-empty open ${ G\subset U}$ with $(\forall x, y\in G) (|f(x)-f(y)|<\eps)$.
\item $f$ is \emph{pointwise discontinuous} if for any $x\in [0,1], k\in \N$ there is $y\in B(x, \frac{1}{2^{k}})$ such that $f$ is continuous at $y$.
\item $f$ is \emph{locally bounded} if for any $x\in [0,1]$, there is $N\in \N$ such that $(\forall y\in B(x, \frac{1}{2^{N}})\cap [0,1])(|f(y)|\leq N)$.
\end{itemize}
\edefi
As to notations, a common abbreviation is `usco', `lsco', and `$BV$' for the first three items.  
Cliquishness and pointwise discontinuity on the reals are equivalent, the non-trivial part of which was already observed by Dini (\cite{dinipi}*{\S 63}).
Moreover, if a function has a certain weak continuity property at all reals in $[0,1]$ (or its intended domain), we say that the function has that property.  
The fundamental theorem about $BV$-functions was proved already by Jordan in \cite{jordel}*{p.\ 229}.
\begin{thm}[Jordan decomposition theorem]\label{drd}
A $BV$-function $f : [0, 1] \di \R$ is the difference of  two non-decreasing functions $g, h:[0,1]\di \R$.
\end{thm}
\noindent
Theorem \ref{drd} has been studied in RM via second-order codes (\cite{nieyo}).  For a $BV$-function $f:[0, 1]\di \R$,  the \emph{total variation} is defined as
\be\label{tombz}\textstyle
V_{0}^{1}(f):=\sup_{0\leq x_{0}< \dots< x_{n}\leq 1}\sum_{i=0}^{n} |f(x_{i})-f(x_{i+1})|.
\ee
Thirdly, the following sets are often crucial in proofs relating to discontinuous functions, as can be observed in e.g.\ \cite{voordedorst}*{Thm.\ 0.36}.
\bdefi
The sets $C_{f}$ and $D_{f}$ \(if they exist\) respectively gather the points where $f:\R\di \R$ is continuous and discontinuous.
\edefi
One problem with the sets $C_{f}, D_{f}$ is that the definition of continuity involves quantifiers over $\R$.  
In general, deciding whether a given $\R\di \R$-function is continuous at a given real, is as hard as $\exists^{3}$ from Section \ref{lll}.
For these reasons, the sets $C_{f}, D_{f}$ only exist in strong systems.
A solution is discussed in just below.

\smallskip

In this section, we introduce \emph{oscillation functions} and provide some motivation for their use.
We have previously studied usco, Baire 1, Riemann integrable, and cliquish functions using oscillation functions (see \cite{samcsl23, samBIG2}). 
As will become clear, such functions are generally necessary for our RM-study.  

\smallskip

Fourth, the study of regulated functions in \cites{dagsamXI, dagsamXII, dagsamXIII, samBIG} is 
really only possible thanks to the associated left- and right limits (see Definition \ref{flung}) \emph{and} the fact that the latter are computable in $\exists^{2}$.  
Indeed, for regulated $f:\R\di \R$, the formula 
\be\label{figo}\tag{\textup{\textsf{C}}}
\text{\emph{ $f$ is continuous at a given real $x\in \R$}}
\ee
involves quantifiers over $\R$ but is equivalent to the \emph{arithmetical} formula $f(x+)=f(x)=f(x-)$.  
In this light, we can define the set $D_{f}$ -using only $\exists^{2}$- and proceed with the usual (textbook) proofs.  
%
An analogous approach, namely the study of usco, Baire 1, Riemann integrable, and cliquish functions, was used in \cite{samcsl23, samBIG2}.  
To this end, we used \emph{oscillation functions} as in Definition \ref{oscfn}.  
We note that Riemann, Ascoli, and Hankel already considered the notion of oscillation in the study of Riemann integration (\cites{hankelwoot, rieal, ascoli1}), i.e.\ there is ample historical precedent. 
\bdefi[Oscillation functions]\label{oscfn}
For any $f:\R\di \R$, the associated \emph{oscillation functions} are defined as follows: $\osc_{f}([a,b]):= \sup _{{x\in [a,b]}}f(x)-\inf _{{x\in [a,b]}}f(x)$ and $\osc_{f}(x):=\lim _{k \di \infty }\osc_{f}(B(x, \frac{1}{2^{k}}) ).$
\edefi
 We  stress that $\osc_{f}:\R\di \R$ is \textbf{only}\footnote{To be absolutely clear, the notation `$\osc_{f}$' and the appearance of $f$ therein in particular, is purely symbolic, i.e.\ we do not make use of the fourth-order object $\lambda f.\osc_{f}$ in this paper.} 
 a third-order object, as clearly indicated by its type.   
Now, our main interest in Definition \ref{oscfn} is that \eqref{figo} is equivalent to the \emph{arithmetical} formula $\osc_{f}(x)=0$, assuming the latter function is given.  
Hence, in the presence of $\osc_{f}:\R\di \R$ and $\exists^{2}$, we can define $D_{f}$ and proceed with the usual (textbook) proofs, which is the approach we \emph{often} took in \cite{samcsl23, samBIG2}.   
Indeed, one can generally avoid the use of oscillation functions for usco functions. 

\subsubsection{On countable sets}\label{couse}
In this section, we discuss the correct definition of countable set for higher-order RM, where we hasten to add that `correct' is only meant to express `yields many equivalences over weak systems like the base theory'. 

\smallskip

First of all, the correct choice of definition for mathematical notions is crucial to the development of RM, as can be gleaned from the following quote.  
\begin{quote}
Under the old definition [of real number in \cite{simpson3}], it would be consistent with $\RCA_{0}$ that there exists a sequence of real numbers $(x_{n})_{n\in \N}$ such that $(x_{n}+\pi)_{n\in \N}$ is not a sequence of real numbers. We thank Ian Richards for pointing out this defect of the old definition. Our new definition [of real number in \cite{earlybs}], given above, is adopted in order to remove this defect. All of the arguments and results of \cite{simpson3}
remain correct under the new definition. (\cite{earlybs}*{p.\ 129})
\end{quote}
In short, the early definition of `real number' from \cite{simpson3} was not suitable for the development of RM, highlighting the importance of the right choice of definition.  
Similar considerations exist for the definition of continuous function in constructive mathematics (see \cites{waaldijk, vandebrug}), i.e.\ this situation is not unique to RM. 

\smallskip

Secondly, we focus on identifying the correct definition for `countable set'.  
Now, going back to Hankel (\cite{hankelwoot}), the sets $C_{f}$ and $D_{f}$ of (dis)continuity points of $f:[0,1]\di \R$ play a central role in real analysis as is clear from e.g.\ the Vitali-Lebesgue theorem which expresses Riemann integrability in terms of $C_{f}$.  For regulated $f:[0,1]\di \R$, the set of discontinuity points satisfies $D_{f}=\cup_{k\in \N} D_{k}$ for
\be\label{drux}\textstyle
 D_{k}:=\{ x\in [0,1]: |f(x)-f(x+)|>\frac{1}{2^{k}}\vee |f(x)-f(x-)|>\frac{1}{2^{k}} \},
\ee
where $D_{k}$ is finite via a standard compactness argument.  In this way, $D_{f}$ is countable but we are unable to construct an injection (let alone a bijection) from the former to $\N$ in e.g.\ $\Z_{2}^{\omega}$.      
Hence, we readily encounter countable sets `in the wild', namely $D_{f}$ for regulated $f$, for which the set-theoretic definition based on injections and bijections can apparently not be established in weak logical systems.  
Similarly, while $D_{k}$ is closed, $\Z_{2}^{\omega}$ does not prove the existence of an RM-code for $D_{k}$.  
In this way, theorems about countable (or RM-closed) sets \emph{in the sense of Definition \ref{char}} cannot be applied to $D_{f}$ or $D_{k}$ in fairly strong systems like $\Z_{2}^{\omega}$ and the development of the RM of real analysis therefore seemingly falters. 

\smallskip

The previous negative surprise notwithstanding, \eqref{drux} also provides the solution to our problem: we namely have $D_{f}=\cup_{k\in \N} D_{k}$, i.e.\ the set $D_{f}$ is
\begin{center}
\emph{the union over $\N$ of finite sets}.  
\end{center}
\emph{Moreover}, this property of $D_{f}$ \emph{can} be established in a (rather) weak logical system.  
Thus, we arrive at Definition \ref{hoogzalieleven} which yields \emph{many} equivalences involving the statement \emph{the unit interval is not height-countable} (see Section \ref{bigger} and \cite{samBIG}) on one hand, and basic properties of regulated functions on the other hand.
\begin{defi}\label{hoogzalieleven}
A set $A\subset \R$ is \emph{height-countable} if there is a \emph{height} function $H:\R\di \N$ for $A$, i.e.\ for all $n\in \N$, $A_{n}:= \{ x\in A: H(x)<n\}$ is finite.  
\end{defi}
Height functions can be found in the modern literature \cites{demol,vadsiger, royco,komig,hux}, but also go back to Borel and Drach circa 1895 (see \cites{opborrelen3, opborrelen4, opborrelen5})
Definition~\ref{hoogzalieleven} amounts to `union over $\N$ of finite sets', as is readily shown in $\ACAo$.

\smallskip

Finally, the observations regarding countable sets also apply \emph{mutatis mutandis} to finite sets.  
Indeed, finite as each $D_{n}$ from \eqref{drux} may be, we are unable to construct an injection to a finite subset of $\N$, even assuming $\Z_{2}^{\omega}$.  
By contrast, one readily\footnote{The standard compactness argument that shows that $D_{k}$ is finite as in Def.\ \ref{deadd} goes through in $\ACAo+\QFAC^{0,1}$ by (the proof of) Theorem \ref{tomma}} shows that $D_{n}$ from \eqref{drux} is finite as in Definition \ref{deadd}.  
\begin{defi}[Finite set]\label{deadd}
Any $X\subset \R$ is \emph{finite} if there is $N\in \N$ such that for any finite sequence $(x_{0}, \dots, x_{N})$ of distinct reals, there is $i\leq N$ such that $x_{i}\not \in X$.
\edefi
The number $N$ from Definition \ref{deadd} is called a \emph{size bound} for the finite set $X\subset \R$.
Analogous to countable sets, the RM-study of regulated functions should be based on Definition~\ref{deadd} and \textbf{not} on the set-theoretic definition based on injections/bijections to finite subsets of $\N$ or similar constructs.

\smallskip

In conclusion, Definitions \ref{hoogzalieleven} and \ref{deadd} provide the correct definition of countable set in that they give rise to many RM-equivalences involving basic properties of regulated functions, as established in \cite{sam BIG}. 
By contrast, the set-theoretic definition involving injections and bijections, does not seem to seem suitable for the development of (higher-order) RM.

\section{Equivalences involving weak K\"onig's lemma}\label{worg}
We establish equivalences between $\WKL_{0}$ and approximation theorems for Bernstein polynomials for (dis)continuous functions.  
We sketch similar results for the other Big Five systems $\ACA_{0}$ and $\ATR_{0}$. 
The RM-study of Bernstein polynomial approximation of course hinges on the following set, definable in $\ACAo$ as the defining formula is arithmetical:
\be\label{BF}
B_{f}:=\{ x\in [0,1]: f(x)=\lim_{n\di \infty} B_{n}(f, x)\}.
\ee
Similarly, the left and right limits of regulated functions can be found using $(\exists^{2})$ (see \cite{dagsamXI}*{\S3}), which is another reason why we shall often incorporate the latter in our base theory.  

\smallskip

First of all, we establish the equivalence between $\WKL_{0}$ and the Weierstrass approximation theorem via Bernstein polynomials (\cite{bern1}), 
as the proof is instructive for that of Theorem \ref{echtdecrux} and essential to that of Theorems \ref{LEMCOR} and \ref{flank}.  
\begin{thm}[$\RCAo$]\label{flahy} The following are equivalent:
\begin{itemize}
\item $\WKL_{0}$,
\item For continuous $f:[0,1]\di \R$ and $x\in (0,1)$, we have 
\be\label{nag}
f(x)=\lim_{n\di \infty }B_{n}(f, x),
\ee 
where the convergence is uniform. 
\end{itemize}
\end{thm}
\begin{proof}
First of all, it is well-known that $\RCA_{0}$ can prove basic facts about basic objects like polynomials, as established in e.g.\ \cite{simpson2}*{II.6}.
Similarly, the following can be proved in $\RCA_{0}$, for $p_{n,k}(x):= \binom{n}{k} x^k (1-x)^{n-k}$ and any $x\in [0,1], m\leq n$:
\be\textstyle\label{identitity}
p_{m, n}(x)\geq 0 \textup{ and }  \sum_{k\leq n}\left(x-{k \over n}\right)^{2}p_{n,k}(x)={x(1-x) \over n} \textup{ and }{ \sum _{k\leq n} p_{n, k}(x)=1}
 \ee
using e.g.\ binomial expansion and similar basic properties of polynomials.  

\smallskip

Secondly, by \cite{simpson2}*{IV.2.5}, $\WKL_{0}$ is equivalent to the Weierstrass approximation theorem for \emph{codes} of continuous functions.  
Every code for a continuous function on $[0,1]$ denotes a third-order (continuous) function by \cite{dagsamXIV}*{Theorem 2.2}, which can be established directly
by applying $\QFAC^{1, 0}$ from $\RCAo$ to the formula expressing the totality of the code on $[0,1]$.  Hence, the upward implication follows. 

\smallskip

Thirdly, for the remaining implication, $\WKL_{0}$ implies that continuous functions are \emph{uniformly} continuous on $[0,1]$, both for codes (see \cite{simpson2}*{IV.2.3}) and for third-order functions (see \cite{dagsamXIV}*{Theorem 2.3}).    
Hence, fix continuous $f:[0,1]\di \R$ and $k_{0}\in \N$ and let $N_{0}\in \N$ be such that $(\forall x, y\in [0,1])( |x-y|<\frac{1}{2^{N_{0}}}\di |f(x)-f(y)|<\frac{1}{2^{k_{0}+1}})$.  
Clearly, $f$ is bounded on $[0,1]$, say by $M_{0}\in \N$.  Now fix $x_{0}\in(0,1)$, choose $n\geq 2^{2M_{0}+2N_{0}+k_{0}+1}$, and consider 
\begin{align}
|f(x_{0})- B_{n}(f, x_{0})|
&\textstyle=|f(x_{0})-\sum_{k=0}^{n} f\big(\frac{k}{n}\big) p_{n,k}(x_{0}) |\notag\\
&\textstyle=|\sum_{k=0}^{n}\big( f(x_{0})-f\big(\frac{k}{n}\big)\big) p_{n,k}(x_{0}) |\notag\\
&\textstyle\leq \sum_{k=0}^n \big| f\big(\frac{k}{n}\big) -f(x_{0}) \big| p_{n,k}(x_{0})\notag\\
&\textstyle=\sum_{i=0,1} \sum_{k\in A_{i}} \big| f\big(\frac{k}{n}\big) -f(x_{0}) \big| p_{n,k}(x_{0}),\label{wups2}
\end{align}
where $A_{0}:=\{k\leq n: |x_{0}-\frac{k}{n}|\leq \frac{1}{2^{N_{0}}} \}$ and $A_{1}:=\{0, 1, \dots n\}\setminus A_{0} $, and where the second equality follows by the final formula in \eqref{identitity}.
The sets $A_{i}$ ($i=0,1$) exist in $\RCA_{0}$ by \emph{bounded comprehension} (see \cite{simpson2}*{X.4.4}).
Now consider
\[\textstyle
\sum_{k\in A_{0}} \big| f\big(\frac{k}{n}\big) -f(x_{0}) \big| p_{n,k}(x_{0})\leq \frac{1}{2^{k_{0}+1}}\sum_{k\in A_{0} }p_{n,k}(x_{0})\leq \frac{1}{2^{k_{0}+1}}\sum_{k\leq n}p_{n,k}(x_{0})= \frac{1}{2^{k_{0}+1}},
\]
where the final equality follows by \eqref{identitity}.  For the other sum in \eqref{wups2}, we have 
\[\textstyle
\sum_{k\in A_{1}} \big| f\big(\frac{k}{n}\big) -f(x_{0}) \big| p_{n,k}(x_{0})\leq  2M_{0} \sum_{k\in A_{1}} \frac{(x_{0}-\frac{k}{n})^{2}}{1/2^{2N_{0}}} p_{n,k}(x_{0})  \leq 2M_{0} \sum_{k\leq n} \frac{(x_{0}-\frac{k}{n})^{2}}{1/2^{2N_{0}}} p_{n,k}(x_{0})  
\]
where $k\in A_{1}$ implies $\frac{(x_{0}-\frac{k}{n})^{2}}{1/2^{2N_{0}}}\geq 1$ by definition and where $p_{n,k}(x_{0})\geq 0$ is also used.  
Now apply the second formula from \eqref{identitity} to the final formula in the previous centred equation to obtain 
\[\textstyle
\sum_{k\in A_{1}} \big| f\big(\frac{k}{n}\big) -f(x_{0}) \big| p_{n,k}(x_{0})\leq 2M_{0}2^{2N_{0}} \frac{x_{0}(1-x_{0})}{n} \leq \frac{1}{2^{k_{0}+1}}.
\]
Thus, we have obtained $|f(x_{0})- B_{n}(f, x_{0})|\leq \frac{1}{2^{k_{0}}}$ and uniform convergence. 
\end{proof}
Next, we wish to generalise the previous theorem to discontinuous functions, for which the following theorem from the literature (see e.g.\ \cite{overdeberg}*{Theorem 5.1} or \cite{klo}*{\S4, p.\ 68}) is essential, namely to Theorems \ref{LEMCOR} and \ref{flank}. 
\begin{thm}[$\ACAo$]\label{echtdecrux}
For any $f:[0,1]\di [0,1]$ and $x\in (0,1)$ such that $f(x+)$ and $f(x-)$ exist\footnote{Using $(\exists^{2})$, rational approximation yields $\Phi:[0,1]\di \R^{2}$ such that $\Phi(x)=(f(x+), f(x-))$ in case the limits exist at $x\in [0,1]$.  In this way, `$\lambda x.f(x+)$' makes sense for regulated functions.}, we have
\be\label{franz}\textstyle
\frac{f(x+)+f(x-)}{2}=\lim_{n\di \infty }B_{n}(f, x).
\ee
\end{thm}
\begin{proof}
First of all, the proof of the theorem is similar to that of Theorem \ref{flahy}, but more complicated as $f(x_{0})$ in \eqref{wups2} is replaced by $\frac{f(x_{0}+)+f(x_{0}-)}{2}$ in \eqref{kyt}.  
In particular, \eqref{kyt} involves more sums than \eqref{flahy}, but the only `new' part is to show that \eqref{dogch} becomes arbitrarily small.  An elementary proof of this fact is tedious but straightforward.  For this reason, we have provided
a sketch with ample references.     

\smallskip

Secondly, fix $k_{0}\in \N$ and $x_{0}\in \R$ such that the left and right limits $f(x_{0}-)$ and $f(x_{0}+)$ exist.  By definition, there is $N_{0}\in \N$ such that for any $y, z\in [0,1]$:
\[\textstyle
x_{0}-\frac{1}{2^{N_{0}}}<z  <x_{0}<y <x_{0}+\frac{1}{2^{N_{0}}}\di \big[ |f(x_{0}+)-f(y)|< \frac{1}{2^{k_{0}+1}} \wedge |f(x_{0}-)-f(z)|< \frac{1}{2^{k_{0}+1}}\big].
\]
Note that increasing $N_{0}$ does not change the previous property.
Now use $(\exists^{2})$ to define $A_{0}:=\{ k\leq n  : x_{0}\leq k/n <x_{0}+\frac{1}{2^{N_{0}}}\}$, $A_{1}:=\{ k\leq n  : x_{0}-\frac{1}{2^{N_{0}}}\leq k/n <x_{0}\}$, and $A_{2}:= \{k\leq n:|x_{0}-k/n|\geq \frac{1}{2^{N_{0}}}  \}  $ and consider:  
\begin{align}\textstyle
B_{n}(f, x_{0})-\frac{f(x_{0}+)+f(x_{0}-)}{2}
&\textstyle=\sum_{k=0}^n \big(f\big(\frac{k}{n}\big) -\frac{f(x_{0}+)+f(x_{0}-)}{2} \big )p_{n,k}(x_{0})\notag\\
&\textstyle=\sum_{i=0}^{2}\sum_{k\in A_{i}} \big(f\big(\frac{k}{n}\big) -\frac{f(x_{0}+)+f(x_{0}-)}{2} \big)p_{n,k}(x_{0}),\label{kyt}
\end{align}
where the first equality follows by the final formula in \eqref{identitity} for $x=x_{0}$.  

\smallskip

Thirdly, we note that $k\in A_{2}$ implies $\frac{(x_{0}-\frac{k}{n})^{2}}{1/2^{2N_{0}}}\geq 1$ by definition, yielding
\be\label{dornik}\textstyle
\sum_{k\in A_{2}}p_{n,k}(x_{0})\leq \sum_{k\in A_{2}} \frac{(x_{0}-\frac{k}{n})^{2}}{1/2^{2N_{0}}} p_{n,k}(x_{0})\leq  \sum_{k\leq n} \frac{(x_{0}-\frac{k}{n})^{2}}{1/2^{2N_{0}}} p_{n,k}(x_{0}),
\ee
where the second inequality holds since $p_{n, k}(x_{0})\geq 0$ by \eqref{identitity}.  
Now apply the second formula in \eqref{identitity} to \eqref{dornik} to obtain $\sum_{k\in A_{2}}p_{n,k}(x_{0})\leq 2^{2N_{0}} \frac{x_{0}(1-x_{0})}{n} $. 
Since $f$ is assumed to be bounded on $[0,1]$, we have for $n\geq 2^{2N_{0}+k_{0}+1}$:
\be\label{draai}\textstyle
|\sum_{k\in A_{2}} \big(f\big({k}/{n}\big) -\frac{f(x_{0}+)+f(x_{0}-)}{2} \big)p_{n,k}(x)|<\frac{1}{2^{k_{0}+1}},
\ee 
Fourth, we consider another sum from \eqref{kyt}, namely the following:
\begin{align}\textstyle
&\textstyle\sum_{k\in A_{0}} \big(f\big(\frac{k}{n}\big) -\frac{f(x_{0}+)+f(x_{0}-)}{2} \big)p_{n,k}(x_{0})\notag\\
&=\textstyle\sum_{k \in A_0} (f(\frac{k}{n})-f(x_0+))p_{n,k}(x_0)+\frac{f(x_0+)-f(x_0-)}{2}(\sum_{k \in A_0}p_{n,k}(x_0)\big).\label{frok}
\end{align}
By the choice of $N_{0}$, the first sum in \eqref{frok} satisfies:
\begin{align}
\textstyle |\sum_{k \in A_0} (f(\frac{k}{n})-f(x_0+))p_{n,k}(x_0)|&\leq \textstyle \sum_{k \in A_0} |f(\frac{k}{n})-f(x_0+)|p_{n,k}(x_0)\notag\\
&\textstyle\leq   \frac{1}{2^{k_{0}+1}}\sum_{k\in A_{0}}p_{k, n}(x_{0})\leq \frac{1}{2^{k_{0}+1}},\label{frok2}
\end{align}
where the final step in \eqref{frok2} follows from the final formula in \eqref{identitity}.  The same, namely a version of \eqref{frok} and \eqref{frok2}, holds \emph{mutatis mutandis} for $A_{1}$.   
Thus, consider \eqref{dogch} which consists of the second sum of \eqref{frok} and the analogous sum for $A_{1}$:
\be\label{dogch}\textstyle
\frac{f(x_0+)-f(x_0-)}{2}(\sum_{k \in A_0}p_{n,k}(x_0) - \sum_{k\in A_{1}}p_{n,k}(x_0) ).  
\ee
Now, $p_{n,k}(x)$ is well-known as the \emph{binomial distribution} and the former's properties are usually established in probability theory via conceptual results like the \emph{central limit theorem}.  
In particular, for large enough $N_{0}$ and associated $n$, the sums $\sum_{k \in A_{i}}p_{n,k}(x_0)$ for $i=0,1$ from \eqref{dogch} can be shown to be arbitrarily close to $\frac12$. 
In light of \eqref{dogch}, a proof (in $\ACAo$) of this limiting behaviour of $\sum_{k \in A_{i}}p_{n,k}(x_0)$ (for $i=0,1$) establishes \eqref{franz} and the theorem.

\smallskip

Finally, an \emph{elementary} proof of the limit behaviour of $\sum_{k \in A_{i}}p_{n,k}(x_0)$ ($i=0, 1$) proceeds along the following lines, based on the \emph{de~Moivre-Laplace theorem} which is apparently a predecessor to the central limit theorem.  
\begin{itemize}
\item First of all, \emph{Stirling's formula} provides approximations to the factorial $n!$.  There are \emph{many} elementary proofs of this formula (see e.g.\ \cite{feller2, impens0, weinbor, ching}).
\item Secondly,  the {de~Moivre-Laplace theorem} states approximations to the sum $\sum_{k=k_{1}}^{k_{2}}p_{n,k}(x)$
in terms of the Gaussian integral $\int e^{\frac{-t^{2}}{2}}dt $ (\cite{papoe}).  There are elementary proofs of this approximation using (only) {Stirling's formula}, namely \cite{feller}*{VII.3, p.\ 182} or \cite{ching}*{Theorem 6, p.\ 228}.
\item Thirdly, applying the previous for $\sum_{k \in A_{i}}p_{n,k}(x_0)$ ($i=0, 1$), one observes that the only non-explicit term in the latter is $\int_{0}^{+\infty} e^{\frac{-t^{2}}{2}}dt =\frac{\sqrt{2\pi}}{2}$; the former sum is therefore arbitrarily close to $\frac{1}{2}$ for $N_{0}$ and $n$ large enough. 
\end{itemize}
It is a tedious but straightforward verification that the above proofs can be formalised in $\ACAo$ (and likely $\RCA_{0}$).  
Alternatively, one establishes basic properties of the $\Gamma$-function (see \cite{weinbor}) and derives the de Moivre-Laplace theorem (\cite{papoe}).
\end{proof}
The final part of the previous proof involving $\sum_{k \in A_{i}}p_{n,k}(x_0)$, amounts to the special case of the theorem for (a version of) of the Heaviside function as follows:
\[
H(x):=
\begin{cases}
1 & x\geq 0 \\
0 & \textup{otherwise}
\end{cases}. 
\]   
Building on the previous theorem, we can now develop the RM-study of approximation theorem for discontinuous functions.   
\begin{thm}[$\RCAo$]\label{LEMCOR} The following are equivalent, where $I\equiv [0,1]$.
\begin{enumerate}
\renewcommand{\theenumi}{\alph{enumi}}
\item $\WKL_{0}$.\label{itema}
\item For any function $f:[0,1]\di \R$, $f$ is continuous on $(0,1)$ \emph{if and only if} the equation \eqref{nag} holds \textbf{uniformly} on $(0,1)$.\label{itemb}
\item For any cadlag $f:I\di \R$ and $x\in (0,1)$, we have \label{itemc}
\be\label{franz2}\textstyle
\frac{f(x+)+f(x-)}{2}=\lim_{n\di \infty }B_{n}(f, x),
\ee
with uniform convergence if $C_{f}=I$.
\item For any $U_{0}$-function $f:I\di \R$ and any $x\in (0,1)$, \eqref{franz2} holds, with uniform convergence if $C_{f}=I$.\label{itemd}
\end{enumerate}
If we additionally assume $\QFAC^{0,1}$, the following are equivalent to $\WKL_{0}$.  
\begin{enumerate}
\renewcommand{\theenumi}{\alph{enumi}}
\setcounter{enumi}{4}
\item For any regulated $f:I\di \R$ and $x\in (0,1)$, \eqref{franz2} holds, with uniform convergence if $C_{f}=I$.\label{iteme}
\item For any locally bounded $f:I\di \R$ and any $x\in (0,1)$ such that $f(x+)$ and $f(x-)$ exist, \eqref{franz2} holds, with uniform convergence if $C_{f}=I$.\label{itemf}
\item The previous item restricted to \textbf{any} function class containing $C([0,1])$. \label{itemg}
\end{enumerate}
\end{thm}
\begin{proof}
First of all, to derive item \eqref{itemb} from $\WKL_{0}$, let $f:[0,1]\di \R$ be such that \eqref{franz2} holds uniformly.  
Now fix $k_{0}\in \N$ and let $N_{0}\in \N$ be such that for $n\geq N_{0}$ and $x\in (0,1)$, $|B_{n}(f, x)-f(x)|<\frac{1}{2^{k_{0}}}$.  
Then $B_{N_{0}}(f, x)$ is uniformly continuous, say with modulus $h$ (\cite{simpson2}*{IV.2.9}).  Fix $x, y\in (0,1)$ with $|x-y|<\frac{1}{2^{h(k_{0})}}$ and consider
\[\textstyle
|f(x)-f(y)|\leq |f(x)-B_{N_{0}}(f, x)|+ |B_{N_{0}}(f, x)-B_{N_{0}}(f, y)|+|B_{N_{0}}(f, y)-f(y)|\leq \frac{3}{2^{k_{0}}}.
\]
To show that item \eqref{itema} follows from items \eqref{itemb}-\eqref{itemg}, note that each of the latter implies the second item from Theorem \ref{flahy}, i.e.\ $\WKL_{0}$ follows. 

\smallskip

Secondly, to show that item \eqref{itema} implies items \eqref{itemb}-\eqref{itemg}, we invoke the law of excluded middle as in $(\exists^{2})\vee \neg(\exists^{2})$.  
In case $\neg(\exists^{2})$, all functions on the reals are continuous by \cite{kohlenbach2}*{Prop.~3.12}.  
In this case, $\WKL_{0}$ implies the other items from the theorem by Theorem \ref{flahy}. 
In case $(\exists^{2})$ holds, items \eqref{itemb}-\eqref{itemg} follow by Theorem~\ref{echtdecrux}, assuming we can provide an upper bound to the functions at hand.
To this end, if $f:[0,1]\di \R$ is unbounded, apply $\QFAC^{0,1}$ to the formula
\be\label{frobg}
(\forall n\in \N)(\exists x\in [0,1])(|f(x)|>n), 
\ee
yielding a sequence $(x_{n})_{n\in \N}$ such that $|f(x_{n})|>n$ for all $n\in \N$; the latter sequence has a convergent sub-sequence, say with limit $y\in [0,1]$, by sequential compactness (\cite{simpson2}*{III.2}). Clearly, $f$ is not locally bounded at $y$ and hence also not regulated.
We note that for unbounded cadlag or $U_{0}$-functions, \eqref{frobg} reduces to $(\forall n\in \N)(\exists q\in [0,1]\cap \Q)(|f(q)|>n)$, i.e.\ we can use $\QFAC^{0,0}$ (included in $\RCAo$).  
\end{proof}
The `excluded middle trick' from the previous proof should be used sparingly as some of our mathematical notions do not make much\footnote{Following Definition \ref{char}, the unit interval is not a set in $\RCAo$; a more refined framework for the study of open sets may be found in \cite{dagsamXVI}.} sense in $\RCAo$.  
Using suitable modulus functions (for the definition of $f(x+)$ and $f(x-)$), one could express uniform convergence of \eqref{franz2} for regulated functions; this
does not seem to yield very elegant results, however. 

\smallskip

Next, we briefly treat equivalences involving $\ACAo$ (and $\ATR_{0}$) and approximation theorems via Bernstein polynomials.
Note that the functions $g, h$ in the second item of Theorem \ref{flank} can be discontinuous and recall the set from \eqref{BF}.  
For the final item, there are \emph{many} function classes between $BV$ and regulated, like the functions of bounded Waterman variation (see \cite{voordedorst}).
\begin{thm}[$\RCAo$]\label{flank} The following are equivalent to $\ACA_{0}$
\begin{enumerate}
\renewcommand{\theenumi}{\alph{enumi}}
\item \(Jordan\) For continuous $f:[0,1]\di \R$ in $BV$, there are continuous and non-decreasing $g, h:[0,1]\di \R$ such that $f=g-h$.\label{va}
\item For continuous $f:[0,1]\di \R$ in $BV$, there are non-decreasing $g, h:[0,1]\di \R$ such that $f=g-h$ and $B_{g}=B_{h}=[0,1]$. \label{vb}
\item For continuous $f:[0,1]\di \R$ in $BV$, there are non-decreasing $g, h:[0,1]\di \R$ such that $f=g-h$ and $B_{g}$ and $B_{h}$ are dense in $[0,1]$. \label{vc}
\item For continuous $f:[0,1]\di \R$ in $BV$, there are non-decreasing $g, h:[0,1]\di \R$ such that $f=g-h$ and $B_{g}$ and $B_{h}$ have measure $1$. \label{vd}
\item For continuous $f:[0,1]\di \R$ in $BV$, there are non-decreasing $g, h:[0,1]\di \R$ such that $f=g-h$ and $B_{g}$ and $B_{h}$ are non-enumerable.\label{ve}
\end{enumerate}
\end{thm}
\begin{proof}
The equivalence between $\ACA_{0}$ and item \eqref{va} is proved in \cite{nieyo} \emph{for RM-codes}.  
Now, RM-codes for continuous functions denote third-order functions by \cite{dagsamXIV}*{Theorem 2.2}, working in $\RCAo$.
Moreover, a continuous and non-decreasing function on $[0,1]$ is determined by the function values on $[0,1]\cap \Q$, i.e.\ one readily obtains an RM-code for such functions in $\RCAo$. 
In this way, item \eqref{va} is equivalent to $\ACA_{0}$ as well.  Item \eqref{vb} follows from item \eqref{va} by Theorem~\ref{echtdecrux}. 
To show that item \eqref{vb} implies item \eqref{itema}, invoke the law of excluded middle as in $(\exists^{2})\vee \neg(\exists^{2})$.  
In case $(\exists^{2})$ holds, $\ACA_{0}$ and item \eqref{va} is immediate.  In case $\neg(\exists^{2})$ holds, all functions on $\R$ are continuous (\cite{kohlenbach2}*{Prop.~3.12}) and item \eqref{va} trivially follows from item \eqref{vb}.  The other items are treated in (exactly) the same way.  
\end{proof}
We could also generalise Theorem \ref{flank} using \emph{pseudo-monotonicity} (\cite{voordedorst}*{Def.~1.14}), originally introduced by Josephy (\cite{josi}) as the largest class such that composition with $BV$ maps to $BV$.
The RM of $\ATR_{0}$ as in \cite{dagsamXIV}*{Theorem 2.25} includes the Jordan decomposition theorem restricted to $BV$-functions with an \emph{arithmetical} graph.
One readily shows that $\ATR_{0}$ is equivalent to the third-to-fifth items in Theorem~\ref{flank} with `continuous' removed and restricted to `arithmetical $f:[0,1]\di \R$'.

\smallskip

Next, we establish an equivalence for $(\exists^{2})$ involving Bernstein polynomials.  
We note that `splittings' as in Theorem \ref{kiolp} are rare in second-order RM, but not in higher-order RM, as studied in detail in \cite{samsplit}.
One can prove that $\WKL_{0}$ in the theorem cannot be replaced by \emph{weak weak K\"onig's lemma} (see e.g.\ \cite{simpson2}*{X.1}).  
\begin{thm}[$\RCAo$]\label{kiolp} 
The following are equivalent to $(\exists^{2})$:
\begin{enumerate}
\renewcommand{\theenumi}{\alph{enumi}}
\item There is $f:[0,1]\di [0,1]$ and $x\in (0,1)$ such that $f(x)\ne \lim_{n\di \infty}B_{n}(f, x)$.\label{ea}
\item $\WKL_{0}$ \(or $\ACA_{0}$\) plus: there is $f:[0,1]\di \R$ and $x\in (0,1)$ such that $f(x)\ne \lim_{n\di \infty}B_{n}(f, x)$. \label{eb}
\item $\WKL_{0}$ \(or $\ACA_{0}$\) plus: there is $f:[0,1]\di \R$ that is not bounded \(or: not Riemann integrable, or: not uniformly continuous\).\label{ec}
\end{enumerate}
We cannot remove $\WKL_{0}$ from the second or third item. 
\end{thm}
\begin{proof}
That $(\exists^{2})$ implies items \eqref{ea} and \eqref{eb} follows by applying Theorem~\ref{echtdecrux} to the (suitably modified) Heaviside function. 
Item \eqref{ec} follows from $(\exists^{2})$ by considering e.g.\ Dirichlet's function $\mathbb{1_{Q}}$ for the Riemann integrable case.
Now let $f:[0,1]\di [0,1]$ be as in item \eqref{ea} and use Theorem \ref{echtdecrux} to conclude that $f$ is discontinuous; then $(\exists^{2})$ follows by \cite{kohlenbach2}*{Prop.\ 3.12}.

\smallskip

Next, assume item \eqref{eb} and suppose $f:[0,1]\di \R$ as in the latter is continuous on $[0,1]$.   Then $\WKL_{0}$ implies that $f$ is bounded (\cite{dagsamXIV}*{Theorem 2.8}); now use Theorem~\ref{echtdecrux} to obtain a contradiction, i.e.\ $f$ must be discontinuous and $(\exists^{2})$ follows as before.  That item \eqref{ec} implies $(\exists^{2})$ follows in the same way.  
Finally, observe that by Theorem~\ref{flahy}, $\RCAo+\neg\WKL_{0}$ proves: \emph{there is $f:[0,1]\di \R$ and $x\in (0,1)$ such that $f(x)\ne \lim_{n\di \infty}B_{n}(f, x)$}.  
Since $\neg\WKL_{0}\di \neg(\exists^{2})$, the final sentence follows and we are done. 
\end{proof}
Finally, the previous results are not unique: \cite{boja}*{Theorem 2} expresses that for $f:[-1, 1]\di \R$ of bounded variation and $x_{0}\in (-1, 1)$, the lim sup and lim inf of 
the Hermite-Fej\'er polynomial $H_{n}(f, x)$ is some (explicit) term involving $f(x_{0}+)$ and $f(x_{0}-)$; this term reduces to $f(x)$ in case $f$ is continuous at $x$.  The proof of this convergence result is moreover lengthy but straightforward, i.e.\ readily formalised in $\ACAo$.  Perhaps surprisingly, the general case essentially reduces to the particular case of the Heaviside function, like in Theorem \ref{echtdecrux}.  
A more complicated, but conceptually similar, approximation result may be found in \cite{kommaar}.  Moreover, the \emph{Bohman-Korovkin theorem} (\cite{lorre2}) suggests near-endless variations of Theorem \ref{LEMCOR}.

\section{Equivalences involving new Big systems}\label{laslo}
\subsection{Introduction}\label{mintro}
In the below sections, we establish equivalences between the new Big systems from \cites{dagsamXI, samBIG, samBIG2} and properties of Bernstein polynomials for (dis)continuous functions, as sketched in Section \ref{sintro}. 
\begin{itemize}
\item The \emph{uncountability of $\R$} is equivalent to the statement that for regulated functions, the Bernstein polynomials converge to the function value for at least one real (Section \ref{bigger}).
\item The \emph{enumeration principle} $\enum$ for countable sets is equivalent to the statement that for regulated functions, the Bernstein polynomials converge to the function value for all reals outside of a given sequence (Section \ref{senum}).
\item The \emph{pigeon-hole principle for measure} is equivalent to the statement that for Riemann integrable functions, the Bernstein polynomials converge to the function value \emph{almost everywhere} (Section \ref{piho}).
\item The \emph{Baire category theorem} is equivalent to the statement that for semi-continuous functions, the Bernstein polynomials converge to the function value for at least one real (Section \ref{BKT}).
\end{itemize}
The second item provides interesting insights into the coding practice of RM:  while the Banach space $C([0,1])$ can be given a code in $\RCAo+\WKL$, the relatively powerful principle $\enum$ is required to code the Banach space of regulated functions by item \eqref{LBQ} of Theorem \ref{tach}.

\smallskip

Now, by the results in \cite{dagsamX, dagsamXI}, the relatively strong system $\Z_{2}^{\omega}+\QFAC^{0,1}$ cannot prove the uncountability of $\R$ formulated as follows
\begin{center}
\textsf{NIN}$_{[0,1]}$: there is no injection from $[0,1]$ to $\N$,
\end{center}
where $\Z_{2}^{\omega}$ proves the same second-order sentences as \emph{second-order arithmetic} $\Z_{2}$ (see \cite{hunterphd} and Section \ref{lll}); the system $\Z_{2}^{\Omega}$ does prove $\NIN_{[0,1]}$, as many of the usual proofs of the uncountability of $\R$ show.  
The above-itemised principles, i.e.\ the Baire category theorem, the enumeration principle, and the pigeon-hole principle for measure, which are studied in Sections \ref{senum}-\ref{piho}, all imply $\NIN_{[0,1]}$. 
In this way, certain approximation theorems are classified in the Big Five by Theorems~\ref{LEMCOR} and~\ref{flank}, while slight variations or generalisations go far beyond the Big Five in light of Theorems \ref{tomma}, \ref{tach}, \ref{Y}, and \ref{duck555}, but are still equivalent to known principles, in accordance with the general theme of RM. 

\smallskip

Finally, at least two of the above systems yield conservative extensions of $\ACA_{0}$, where $\PHP_{[0,1]}$ expresses that for a sequence of closed sets of measure zero, 
the union also has measure zero (see Section \ref{piho}).  
\begin{thm}\label{weng} ~
\begin{itemize}
\item The system $\ACAo+\NIN_{[0,1]}$ is $\Pi_{2}^{1}$-conservative over $\ACA_{0}$.
\item The system $\ACAo+\PHP_{[0,1]}$ is $\Pi_{2}^{1}$-conservative over $\ACA_{0}$.
\end{itemize}
\end{thm}
\begin{proof}
Kreuzer shows in \cite{kruisje} that $\ACAo+(\blambda)$ is $\Pi_{2}^{1}$-conservative over $\ACA_{0}$, where $(\blambda)$ expresses the existence of the Lebesgue measure $\blambda^{3}$ as a fourth-order functional on $2^{\N}$ and $[0,1]$.  To derive $\NIN_{[0,1]}$ in the former system, let $Y:[0,1]\di \N$ be an injection and derive a contradiction from the sub-additivity of $\blambda$ as follows:  
\[\textstyle
\blambda([0,1])=\blambda(\cup_{n\in \N}  E_{n} )\leq \sum_{n\in \N}\blambda(E_{n})=0
\]
where $E_{n}:=\{ x\in [0,1]:Y(x)=n\}$ is at most a singleton, i.e.\ $\blambda(E_{n})=0$.   The system $\ACAo+(\blambda)$ trivially proves $\PHP_{[0,1]}$.  
\end{proof}

\subsection{The uncountability of the reals}\label{bigger}
In this section, we establish equivalences between approximation theorems involving Bernstein polynomials and the uncountability of the reals.  
Our results also improve the base theory used in \cite{samBIG}.   

\smallskip

First of all, we will study the uncountability of the reals embodied by the following principle, motivated by the observations in Section \ref{couse}.  
As discussed in Section~\ref{mintro}, $\Z_{2}^{\omega}+\QFAC^{0,1}$ does not\footnote{Note that an injection is a special kind of height function, i.e.\ $\NIN_{\alt}$ implies that there is no injection from $[0,1]$ to $\N$.} prove $\NIN_{\alt}$, and the same for the equivalent approximation theorems in Theorem \ref{tomma}.
\begin{princ}[$\NIN_{\alt}$]
The unit interval is not height-countable.  
\end{princ}
This principle was first introduced in \cite{samBIG} where \emph{many} equivalences are established, mainly for basic properties of regulated functions
and related classes.

\smallskip

Secondly, we have the following theorem involving equivalences for $\NIN_{\alt}$.  
Item \eqref{KC} is defined using $B_{f}$ from \eqref{BF} and constitutes a variation of Volterra's early theorem from \cite{volaarde2} that there is no $f:\R\di \R$ satisfying $C_{f}=\Q$. 
In light of the first equivalence, the restriction in item \eqref{KBP} is non-trivial.
\begin{thm}[$\ACAo+\QFAC^{0,1}$]\label{tomma}
The following are equivalent.
\begin{enumerate}
\renewcommand{\theenumi}{\alph{enumi}}
\item The uncountability of $\R$ as in $\NIN_{\alt}$.\label{KA}
\item For regulated $f:[0,1]\di \R$, there is $x\in (0,1)\setminus \Q$ where $f$ is continuous.\label{KB}  
\item For regulated and pointwise discontinuous $f:[0,1]\di \R$, there is $x\in (0,1)\setminus \Q$ where $f$ is continuous.\label{KBP}  
\item \(Volterra\) There is no regulated $f:[0,1]\di \R$, such that $B_{f}= \Q$ \label{KC}
\item For regulated $f:[0,1]\di \R$, there is $x\in (0,1)$ where $f$ is continuous.\label{KD}  
\item For regulated $f:[0,1]\di \R$, there is $x\in (0,1)$ where $f(x)=\lim_{n\di \infty} B_{n}(f, x)$.\label{KE}
\item For regulated $f:[0,1]\di \R$, the set $B_{f}$ is not \(height-\)countable.\label{KG}
\end{enumerate}
\end{thm}
\begin{proof}
First of all, \eqref{KA}$\di$\eqref{KB} is proved by contraposition as follows: let $f:[0,1]\di \R$ be regulated and discontinuous on $[0,1]\setminus \Q$.  
In particular, $[0,1]\setminus \Q=D_{f}= \cup_{k\in \N}D_{k}$ where $D_{k}$ is as in \eqref{drux}.  To show that $D_{k}$ is finite, suppose it is not, i.e.\ for any $N\in \N$, there are $x_{0}, \dots x_{N}\in D_{k}$.  
Use $(\exists^{2})$ and $\QFAC^{0,1}$ to obtain a sequence $(x_{n})_{n\in \N}$ in $D_{k}$.  Since sequential compactness follows from $\ACA_{0}$ (\cite{simpson2}*{III.2.2}), the latter sequence has a convergent sub-sequence, say with limit $y\in [0,1]$.  Then either on the left or on the right of $y$, one finds infinitely many elements of the sequence.  Then either $f(y+)$ or $f(y-)$ does not exist, a contradiction, and $D_{k}$ is finite.   

\smallskip

To establish $\neg\NIN_{\alt}$, let $(q_{m})_{m\in \N}$ be an enumeration of $\Q$ without repetitions and define a height function $H:[0,1]\di \N$ for $[0,1]$ as follows:  
$H(x)=n$ in case $x\in D_{n}$ and $n$ is the least such number, $H(x)=m$ in case $x=q_{m}$ otherwise. 
By contraposition and using $\QFAC^{0,1}$, one proves that the union of two finite sets is finite, implying that $D_{k}\cup \{q_{0}, \dots, q_{k}\}$ is finite.
A direct proof in $\ACAo$ is also possible as the second set has an (obvious) enumeration.  Hence, $[0,1]$ is height countable, as required for $\neg\NIN_{\alt}$. 

\smallskip

Secondly, \eqref{KB}$\di$\eqref{KC} is immediate by Theorem \ref{echtdecrux} and the fact that $f(x+)=f(x-)=f(x)$ in case $x\in C_{f}$.   Now assume item \eqref{KC} and suppose $\NIN_{\alt}$ is false, i.e.\ $H:[0,1]\di \N$ is a height function for $[0,1]$. 
Define $g:[0,1]\di [0,1]$ using $(\exists^{2})$: 
\[
g(x):=
\begin{cases}
\frac{1}{2^{H(x)}} & x\not \in \Q  \\
0 & x \in \Q
\end{cases}
\]
To show that $g(x+)=g(x-)=0$ for any $x\in (0,1)$, consider
\be\label{teng}\textstyle
(\forall k\in \N )(\exists N\in \N)(\forall y \in (x-\frac{1}{2^{N}}, x))(|g(y)|<\frac{1}{2^{k}}). 
\ee
Suppose \eqref{teng} is false, i.e.\ there is $k_{0}\in \N$ such that $(\forall N\in \N)(\exists y \in (x-\frac{1}{2^{N}}, x))(|f(y)|\geq \frac{1}{2^{k}})$. 
Modulo the coding of real numbers, apply $\QFAC^{0,1}$ to obtain a sequence of reals $(y_{n})_{n\in \N}$ such that 
\be\label{forion}\textstyle
(\forall N\in \N)( y_{N} \in (x-\frac{1}{2^{N}}, x)\wedge |g(y_{N})|\geq \frac{1}{2^{k_{0}}}).
\ee
Using $\mu^{2}$, we can guarantee that each $y_{n}$ is unique using the (limited) primitive recursion\footnote{Define $h(z):=(\mu m)(z<x-\frac{1}{2^{m}})$, $H(0):=y_{0}$ and $H(n+1)=y_{h(H(n))}$.} available in $\RCAo$.  
Now, by definition, the set $C:=\{x\in [0,1]: H(x)\leq {k_{0}}\}$ is finite, say with upper bound $K_{0}\in \N$.  Apply \eqref{forion} for $N=K_{0}+2$ and note that $y_{0}, \dots, y_{K_{0}+2}$ are in $C$ by the definition of $g$, a contradiction.  Hence, \eqref{teng} is correct, implying that $g$ is regulated.  Now apply item \eqref{KC} to obtain $x_{0}\in (0, 1)\setminus\Q$ such that $g(x_{0})=\lim_{n\di \infty }B_{n}(g, x_{0})$.  
However, $g(x_{0})>0$ and $B_{n}(g, x_{0})=0$, a contradiction, and $\NIN_{\alt}$ must hold.  The equivalence involving item \eqref{KBP} is immediate as $g$ is continuous at every rational point in $[0,1]$, i.e.\ pointwise discontinuous.

\smallskip

Thirdly, we immediately have \eqref{KB}$\di$\eqref{KD}$\di$\eqref{KE} by Theorem \ref{echtdecrux}, and to show that item \eqref{KE} implies $\NIN_{\alt}$, we again proceed by contraposition.  To this end, let $H:\R\di \N$ be a height function for $[0,1]$ and define $h(x):= \frac{1}{2^{H(x)+1}}$.  We now prove prove $\lim_{n\di \infty}B_{n}(h, x)=0$ for all $x\in [0,1]$. 
As Bernstein polynomials only invoke rational function values, $B_{n}(h, x)=B_{n}(\tilde{h}, x)$ for all $x\in [0,1]$ and $n\in \N$, where
\be\label{tepz}
\tilde{h}(x):=
\begin{cases}
0 & \textup{ if }x\not \in \Q\\
\frac{1}{2^{H(x)+1}} & \textup{ otherwise }
\end{cases},
\ee
which is essentially a version of Thomae's function (\cite{thomeke, dagsamXIV}).  
Similar to $g$ above, one verifies that $\tilde{h}$ is continuous on $[0,1]\setminus \Q$ using the fact that $H$ is a height function.  
By Theorem~\ref{echtdecrux}, we have $0=\tilde{h}(x)=\lim_{n\di\infty}B_{n}(\tilde{h}, x)$ for $x\in (0,1)\setminus \Q$.  
By the uniform continuity of the polynomials $p_{k, n}(x)$ on $[0,1]$, we also have $\lim_{n\di\infty}B_{n}(\tilde{h}, q)=0$ for $q\in [0,1]\cap \Q$.
As a result, we obtain $\lim_{n\di \infty}B_{n}(h, x)=0$ for all $x\in [0,1]$, which implies the negations of items \eqref{KD} and \eqref{KE}, as $h(x)>0$ for all $x\in [0,1]$. 
Clearly, item \eqref{KG} implies item \eqref{KE} while the former readily follows from $\NIN_{\alt}$.  Indeed, if item \eqref{KG} is false for some regulated $f:[0,1]\di \R$, then $C_{f}\subset B_{f}=\cup_{n\in \N}E_{n}$ by Theorem \ref{echtdecrux} for a sequence of finite sets $(E_{n})_{n\in \N}$.  Since $D_{f}=\cup_{n\in \N}D_{n}$, $[0,1]$ is height-countable as $D_{n}\cup E_{n}$ is finite by the first paragraph of the proof.  
\end{proof}
We note that the base theory in the previous theorem is much weaker than in \cite{samBIG}. 
Indeed, in the latter, the base theory also contained various axioms governing finite sets, all provable from the induction axiom. 
It seems that the latter can always be replaced by proofs by contradiction involving $\QFAC^{0,1}$.  

\smallskip

It is interesting that we can study items \eqref{KB} and \eqref{KC} without using (basic) real analysis, while items \eqref{KD} and \eqref{KE} seem to (really) require 
the study of the Bernstein approximation of \eqref{tepz}.  Theorem \ref{tomma} goes through for `regulated' replaced by `bounded variation' if we assume e.g.\ a small fragment
of the classical function hierarchy, namely that a regulated function has bounded Waterman variation (\cite{voordedorst}).  

\smallskip

Next, we show that we can weaken the conclusion of item \eqref{KD} in Theorem~\ref{tomma} to the weak continuity notions from Def.\ \ref{KY}; these are a considerable improvement over e.g.\ quasi-continuity in \cite{samBIG} and go back over a hundred years, namely to Young (\cite{jonnypalermo}) and Blumberg (\cite{bloemeken}).  
\bdefi[Weak continuity]\label{KY} For $f:[0,1]\di \R$, we have that
\begin{itemize}
\item $f$ is \emph{almost continuous} \(Husain, see \cites{husain,bloemeken}\) at $x\in [0,1]$ if for any open $G\subset \R$ containing $f(x)$, the set $\overline{f^{-1}(G)}$ is a neighbourhood of $x$,
\item $f$ has the \emph{Young condition} at $x\in [0,1]$ if there are sequences $(x_{n})_{n\in \N}, (y_{n})_{n\in \N}$ on the left and right of $x$ with the latter as limit and $\lim_{n\di \infty}f(x_{n})=f(x)=\lim_{n\di \infty}f(y_{n})$.
\end{itemize}
\edefi\noindent
With these definitions in place, we now have the following corollary to Theorem~\ref{tomma}, which establishes a nice degree of robustness for the RM of $\NIN_{\alt}$.  
\begin{cor}\label{fraaaa}
Theorem \ref{tomma} holds if we replace `continuity' in item \eqref{KD} by `almost continuity' or `the Young condition'.
\end{cor}
\begin{proof}
Assuming $\NIN_{\alt}$, the function $h:[0,1]\di \R$ from the proof of Theorem \ref{tomma} does not satisfy either weak continuity notion anywhere.  
Indeed, the sequences from the Young condition are immediately seen to violate the fact that $H$ from the proof of Theorem \ref{tomma} is a height-function for $[0,1]$.
Similarly, take any $x_{0}\in [0,1]$ and note that for $k_{0}\in \N$ such that $\frac{1}{2^{k_{0}}}< h(x_{0})$, the set $f^{-1}(B(x_{0}, \frac{1}{2^{k_{0}}}))$ is finite, as $H$ is a height-function for $[0,1]$.  
 \end{proof}

\subsection{The enumeration principle}\label{senum}
In this section, we establish equivalences between approximation theorems involving Bernstein polynomials and the enumeration principle for height-countable sets as in Principle \ref{gelum} below.  
Our results significantly improve the base theory used in \cite{dagsamXI} and provide the `definitive' RM of Jordan's decomposition theorem (\cite{jordel}), especially in light of the new 
connection to Helly's selection theorem as in Principle \ref{HEY}. 

\smallskip

First of all, we will study Principle \ref{gelum}, motivated by Section \ref{couse}.  
\begin{princ}[$\enum$]\label{gelum}
A height-countable set in $[0,1]$ can be enumerated.
\end{princ}
We stress that textbooks (see e.g.\ \cite{voordedorst}*{p.\ 28} and \cite{rudin}*{p.\ 97}) generally only prove that certain sets are (height) countable, i.e.\ no enumeration is provided, while the latter is 
readily assumed in other places, i.e.\ $\enum$ is implicit in the mathematical mainstream.
By the results in \cite{dagsamXI}, $\ACAo+\enum$ proves $\ATR_{0}$ and $\FIVE^{\omega}+\enum$ proves $\SIX$, i.e.\ $\enum$ is rather `explosive', in contrast to $\NIN_{[0,1]}$ by Theorem~\ref{weng}. 

\smallskip

A variation of $\enum$ for countable sets, called $\cocode_{0}$, is studied in \cite{dagsamXI} where {many} equivalences are established; the base theories used in \cite{dagsamXI} are however not always elegant.  Now, most proofs in \cite{dagsamXI} go through \emph{mutatis mutandis} for `countable' replaced by `height-countable'; the latter's central role was not known during the writing of \cite{dagsamXI}. We provide some examples in Theorem \ref{tach}, including new results for Bernstein approximation.   
We also need the following principle which is a contraposed version of \emph{Helly's selection theorem} (\cites{hellyeah, naatjenaaien}); the RM of the latter for \emph{codes} of $BV$-functions is studied in e.g.\ \cite{kreupel}.  
\begin{princ}[\textsf{Helly}]\label{HEY}
Let $(f_{n})_{n\in \N}$ be a sequence of $[0,1]\di [0,1]$-functions in $BV$ with pointwise limit $f:[0,1]\di [0,1]$ which is not in $BV$.  Then there is unbounded $g\in \N^{\N}$ such that $g(n)\leq V_{0}^{1}(f_{n})\leq g(n)+1$ for all $n\in \N$,
\end{princ}
Intuitively, \textsf{Helly} is a rather weak statement that significantly simplifies the RM-study of $\enum$ and $\NIN_{\alt}$.  

\smallskip

Next, we establish Theorem \ref{tach} where we note that the base theory is weaker than in \cite{dagsamXI}. 
Indeed, in the latter, the base theory also contains various axioms governing countable sets, mostly provable from the induction axiom. 
\begin{thm}[$\ACAo+\QFAC^{0,1}$]\label{tach} 
The following are equivalent.
\begin{enumerate}
\renewcommand{\theenumi}{\alph{enumi}}
\item The enumeration principle $\enum$.\label{LA}
\item For regulated $f:[0,1]\di \R$, there is a sequence $(x_{n})_{n\in \N}$ enumerating $D_{f}$.\label{LB}  
\item For regulated $f:[0,1]\di \R$ and $p, q\in [0,1]\cap \Q$, $\sup_{x\in [p,q]}f(x)$ exists\footnote{To be absolutely clear, we assume the existence of a `supremum operator' $\Phi:\Q^{2}\di \R$ such that $\Phi(p, q)=\sup_{x\in [p, q]}f(x)$ for all $p, q\in [0,1]\cap \Q$.  For Baire 1 functions, this kind of operator exists in $\ACAo$ by \cite{dagsamXIV}*{\S2}, even for irrational intervals.}.\label{LBQ}  
\item For regulated and pointwise discontinuous $f:[0,1]\di \R$, there is a sequence $(x_{n})_{n\in \N}$ enumerating $D_{f}$.\label{LBP}  
\item For regulated $f:[0,1]\di \R$, there is $(x_{n})_{n\in \N}$ enumerating $[0,1]\setminus B_{f}$.\label{LC}  
\item \(Jordan\) For $f:\R\di \R$ which is in $BV([0, a])$ for all $a>0$, there are monotone $g, h:\R\di \R$ such that $f(x)=g(x)-h(x)$ for $x\geq 0$.\label{LD}
\item The combination of: \label{LE}
\begin{itemize}
\item[(f.1)] Helly's selection theorem as in Principle \ref{HEY}.
\item[(f.2)] \(Jordan\) For $f:[0,1]\di \R$ in $BV$, there are non-decreasing $g, h:\R\di \R$ such that $f(x)=g(x)-h(x)$ for $x\in [0,1]$.
\end{itemize}
\end{enumerate}
\end{thm}
\begin{proof}
The implication \eqref{LA}$\di$\eqref{LB} follows by noting that $D_{k}$ as in \eqref{drux} is finite, as established in the proof of Theorem \ref{tomma}, and applying $\enum$.  
The implication \eqref{LB}$\di$\eqref{LC} follows from Theorem \ref{echtdecrux}.  The implication \eqref{LB}$\di$\eqref{LBQ} is immediate as we can replace the supremum over $p, q$ in $\sup_{x\in [p,q]}f(x)$ by a supremum over $[p, q]\cap (\Q\cup D_{f})$.  

\smallskip

For the implication \eqref{LC}$\di$\eqref{LA}, let $A$ be height-countable, i.e.\ there is $H:\R\di \N$ such that $A_{n}:=\{x\in [0,1]: H(x)<n\}$ is finite.  
Note that we can use $\mu^{2}$ to enumerate $A\cap \Q$, i.e.\ we may assume the latter is empty. 
Now define $g:[0,1]$ as
\be\label{zopi}
g(x):=
\begin{cases}
\frac{1}{2^{n+1}} & \textup{ if $x\in A$ and $n$ is the least natural such that $H(x)<n$}\\
0 & \textup{ if $x\not \in A$}
\end{cases}.
\ee
The function $g$ satisfies $g(x+)=g(x-)=0$ for $x\in (0,1)$, which is proved in exactly the same way as in the proof of Theorem \ref{tomma}.
Let $(x_{n})_{n\in \N}$ be the sequence provided by item \eqref{LC} and note that by Theorem \ref{echtdecrux}, $(\forall n\in \N)(x\ne x_{n} )$ implies $g(x)=\frac{g(x+)+g(x-)}{2}=0$ for any $x\in (0,1)$. 
Hence, $(x_{n})_{n\in \N}$ includes all elements of $A$, and $(\exists^{2})$ can remove all elements not in $A$, as required for $\enum$.  Item \eqref{LB} trivially implies item \eqref{LBP}, while $g$ from \eqref{zopi} is continuous on the rationals and therefore pointwise discontinuous, i.e.\ item \eqref{LBP} also implies item \eqref{LA}.  To show that item \eqref{LBQ} implies $\enum$, apply the former to $g$ as in \eqref{zopi}.  In particular, one enumerates $D_{g}$ using the usual interval-halving technique.      

\smallskip

For the implication \eqref{LA}$\di \eqref{LD}$, let $f$ be as in the former and define $D_{n,k}$ as:
\[\textstyle
 D_{n,k}:=\{ x\in [0,n]: |f(x)-f(x+)|>\frac{1}{2^{k}}\vee |f(x)-f(x-)|>\frac{1}{2^{k}} \},
\]
where we note that $BV$-functions are regulated (using $\QFAC^{0,1}$) by \cite{dagsamX}*{Theorem~3.33}.
As for $D_{k}$ in \eqref{drux}, this set is finite and $D_{f}=\cup_{n, k\in \N}D_{n,k}$ can be enumerated thanks to $\enum$.  
Now consider the \emph{variation function} defined as:
\be\label{tomb}\textstyle
V_{0}^{y}(f):=\sup_{0\leq x_{0}< \dots< x_{n}\leq y}\sum_{i=0}^{n} |f(x_{i})-f(x_{i+1})|, 
\ee
where the supremum is over all partitions of $[0, y]$.  Since we have an enumeration of $D_{f}$, \eqref{tomb} can be defined using $\exists^{2}$ by restricting the supremum to $\Q$ and this enumeration. 
By definition, $g(x):=\lambda x.V_{0}^{y}(f)$ is non-decreasing and the same for $h(x):=g(x)-f(x)$.  
For the reverse implication, let $A$ be height-countable, i.e.\ there is $H:\R\di \N$ such that $A_{n}:=\{x\in [0,1]: H(x)<n\}$ is finite.  
Then let $f:\R\di \R$ be the indicator function of $A_{n}$ on $[0, n]$, which is $BV$ on $[0, n]$ by \cite{dagsamX}*{Theorem 3.33}.
Now apply item \eqref{LD} and note that $(\exists^{2})$ can enumerate $D_{g}$ for monotone $g$ by \cite{dagsamX}*{Theorem 3.33}, i.e.\ $\enum$ now follows. 

\smallskip

To establish item \eqref{LE} using $\enum$, sub-item (f.2) is a special case of item \eqref{LD}.  For sub-item~(f.1), consider Helly's selection theorem, usually formulated as follows. 
\begin{center}
\emph{Let $(f_{n})_{n\in \N}$ be a sequence of $BV$-functions such that $f_{n}$ and $V_{0}^{1}(f_{n})$ are uniformly bounded.  Then there is a sub-sequence $(f_{n_{k}})_{k\in \N}$ with pointwise limit $f\in BV$.}
\end{center}
To establish the centred statement, Helly's original proof from \cite{hellyeah}*{p.\ 287} or \cite{naatjenaaien}*{p.~222} goes through in $\ACAo+\enum$ as follows: for a sequence $(f_{n})_{n\in \N}$ in $BV$ as above, one uses $\enum$ to obtain sequences of monotone functions $(g_{n})_{n\in \N}$ and $(h_{n})_{n\in \N}$ such that $f_{n}=g_{n}-h_{n}$.  
Then $(g_{n})_{n\in \N}$ and $(h_{n})_{n\in \N}$ have convergent sub-sequences with limits $g$ and $h$ that are monotone, which is (even) provable in $\ACAo$. 
Essentially by definition, $f=g-h$ is then the limit of the associated sub-sequence of $(f_{n})_{n\in \N}$.  
To obtain sub-item (f.1), use the contraposition of the centred statement, i.e.\ if the limit function $f$ is not in $BV$, then $(\forall N\in \N)(\exists n)(V_{0}^{1}(f_{n})\geq N )$. 
As in the previous paragraph, $V_{0}^{1}(f_{n})$ can be defined using $\exists^{2}$, given $\enum$.  The function $g\in \N^{\N}$ from the conclusion of sub-item (f.1) is therefore readily obtained (using $\QFAC^{0,0}$ and $\exists^{2}$). 

\smallskip

For the remaining implication \eqref{LE}$\di$\eqref{LA}, let $A$ be height-countable, i.e.\ there is $H:\R\di \N$ such that $A_{n}:=\{x\in [0,1]: H(x)<n\}$ is finite.  
As above, we may assume $A\cap \Q=\emptyset$ as $\exists^{2}$ can enumerate the rationals in $A$.  
Define $f_{n}(x):=\mathbb{1}_{A_{n}}(x)$, which is $BV$ by \cite{dagsamX}*{Theorem 3.33} and note $\lim_{n\di \infty}f_{n}=f$ where $f:=\mathbb{1}_{A}$.  
If $f\in BV$, apply sub-item (f.2) and recall that $(\exists^{2})$ can enumerate $D_{g}$ for monotone $g$ by \cite{dagsamX}*{Theorem 3.33}, i.e.\ $\enum$ now follows. 
If $f\not \in BV$, let $g_{0}\in \N^{\N}$ be the function provided by sub-item (f.1) and note that $V_{0}^{1}(f_{n})\leq g_{0}(n)+1$ implies that $A_{n}$ is finite \emph{and} has size bound $g_{0}(n)+2$.  Now let $\tilde{g}:[0,1]\di \R$ be \eqref{zopi} with `$\frac{1}{2^{n}}$' replaced by `$\frac{1}{2^{n}(g_{0}(n)+2)}$'.  Clearly, for any partition of $[0,1]$, we have $\sum_{i=0}^{n} |\tilde{g}(x_{i})-\tilde{g}(x_{i+1})|\leq \sum_{i=0}^{n}\frac{1}{2^{i}}\leq 2$, i.e.\ $\tilde{g}$ is in $BV$.
Applying sub-item (f.2) to $\tilde{g}$, $\enum$ follows as before, and we are done. 
\end{proof}
Regarding item \eqref{LC} in Theorem \ref{tach}, basic examples show that $C_{f}\ne B_{f}$, while the first equivalence in Theorem \ref{tomma} shows that the restriction in item \eqref{LBP} is non-trivial.
Item \eqref{LBQ} expresses that the Banach space of regulated functions requires $\enum$, in contrast to e.g.\ the Banach space of continuous functions (\cite{simpson2}*{IV.2.13}).
Moreover, the use of monotone functions in Theorem \ref{tach} can be replaced by weaker conditions, as follows.
\begin{cor}
One can replace `monotone' in items \eqref{LD} or \eqref{LE} by:
\begin{itemize}
\item $U_{0}$-function, or:
\item regulated $f:[0,1]\di \R$ such that for all $x\in (0,1)$, we have
\be\label{fereng}\textstyle
|f(x)-\lim_{n\di \infty} B_{n}(f, x)|\leq |\frac{f(x+)-f(x-)}{2}|. 
\ee
\end{itemize}
\end{cor}
\begin{proof}
By \cite{dagsamXIV}*{Theorem 2.16}, $(\exists^{2})$ suffices to enumerate the discontinuity points for functions satisfying the items in the corollary.
\end{proof}
Next, Helly's theorem as in Principle \ref{HEY} is useful in the RM of the uncountability of $\R$, as follows. 
\begin{thm}[$\ACAo+\textsf{\textup{Helly}}$]\label{hex}
The following are equivalent.
\begin{enumerate}
\renewcommand{\theenumi}{\alph{enumi}}
\item The uncountability of $\R$ as in $\NIN_{\alt}$.\label{KAA}
\item \(Volterra\) There is no $f:[0,1]\di \R$ in $BV$, such that $B_{f}= \Q$. \label{KCC}
\item For $f:[0,1]\di \R$ in $BV$, there is $x\in (0,1)$ where $f$ is continuous.\label{KDD}  
\item For $f:[0,1]\di \R$ in $BV$, there is $x\in (0,1)$ where $f(x)=\lim_{n\di \infty} B_{n}(f, x)$.\label{KEE}
\item For $f:[0,1]\di \R$ in $BV$, there is $x\in (0,1)$ such that \eqref{zolk} or \eqref{fereng}.\label{KF}  
\end{enumerate}
\end{thm}
\begin{proof}
That $\NIN_{\alt}$ implies the items \eqref{KCC}-\eqref{KF} follows as in the proof of Theorem~\ref{tomma}, namely by considering $D_{k}$ from \eqref{drux}. 
For $f\in BV$ with variation bounded by $1$, this set has at most $2^{n}$ elements, as each element of $D_{n}$ contributes at least $\frac{1}{2^{n}}$ to the total variation, by definition.
Hence, $D_{k}$ is finite \emph{without} the use of $\QFAC^{0,1}$.  Now, $\NIN_{\alt}$ implies that $D_{f}=\cup_{k\in \N}D_{k}$ is not all of $[0,1]$, i.e.\ $C_{f}\ne \emptyset$
and the other items follow.  
To derive $\NIN_{\alt}$ from items \eqref{KCC}-\eqref{KF}, 
let $H:[0,1]\di \N$ be a height-function for $[0,1]$ and consider the finite sets $A_{n}=\{x\in [0,1]:H(x)<n \}$ and $B_{n}=A_{n}\setminus \Q$.  
As above, $\mathbb{1}_{B_{n}}$ is in $BV$ but $\lim_{n\di \infty}\mathbb{1}_{B_{n}}=\mathbb{1}_{\R\setminus \Q}$ is not in $BV$.
Let $g_{0}\in \N^{\N}$ be the function provided by Helly's selection theorem and let $g_{1}\in \N^{\N}$ be such that $|A_{n}\cap \Q|\leq g_{1}(n)$ for all $n\in \N$, which is readily defined using $\exists^{2}$.  
By definition, $|A_{n}|\leq g_{0}(n)+g_{1}(n)+1$ for all $n\in \N$ and define 
\[\textstyle
g(x):=\frac{1}{2^{H(x)+1}}\frac{1}{(g_{0}(H(x)) +g_{1}(H(x)+1) +1 )}.
\] 
The latter is in $BV$ (total variation at most $1$) and is totally discontinuous, i.e.\ item \eqref{KDD} is false, as required.  The other implications follow in the same way. 
\end{proof}
We believe that Helly's selection theorem as in Principle \ref{HEY} is weak and in particular does not imply $\NIN_{[0,1]}$. 

\smallskip

Finally, we attempt to explain why seemingly related mathematical notions behave so differently in higher-order RM.
\begin{rem}[Second-order-ish functions]\label{second-order-ish}\rm
As discussed in detail in \cite{dagsamXIV}, quasi-continuous and cliquish $[0,1]\di [0,1]$ functions are intimately related from the pov of real analysis.  
Nonetheless, $\RCAo+\WKL_{0}$ proves that the former have a supremum while $\Z_{2}^{\omega}+\QFAC^{0,1}$ cannot prove the existence of a supremum for the latter. 
A similar observation can be made for many pairs of function classes, including cadlag and regulated functions.  

\smallskip

The crucial observation here is that for quasi-continuous and cadlag functions, the function value $f(x)$ for \emph{any} $x\in [0,1]$ is 
determined if we know $f(q)$ for any $q\in [0,1]\cap \Q$, provably in $\ACAo$.   We refer to such function classes as \emph{second-order-ish} as their definition comes with 
an obvious second-order approximation device.  By contrast, cliquish and regulated functions are not determined in this way, i.e.\ they apparently lack the latter device.  

\smallskip

With the gift of hindsight, properties of second-order-ish function classes can be established in $\RCAo$ extended with the Big Five systems, which is one of the main observations of \cite{dagsamXIV}.  
By contrast, basic properties of non-second-order-ish functions can often not be proved in $\Z_{2}^{\omega}$ or even $\Z_{2}^{\omega}+\QFAC^{0,1}$.  Thus, our equivalences seems to be robust as long 
as we stay within the second-order-ish function classes, or dually: within the non-second-order-ish ones.  
 
\end{rem}

\subsection{Pigeon hole principle for measure spaces}\label{piho}
We introduce Tao's {pigeonhole principle for measure spaces} from \cite{taoeps} and obtain equivalences involving the approximation theorem for Riemann integrable functions via Bernstein polynomials.  
The latter is studied in \cite{klo} where it is claimed this result goes back to Picard (\cite{piti}). 

\smallskip

First of all, the pigeonhole principle as in $\PHP_{[0,1]}$ is studied in \cite{samBIG2}, with a number of equivalences for basic properties of Riemann integrable functions.  
\begin{princ}[$\PHP_{[0,1]}$]\label{PHP}
If $ (X_n)_{n \in \N}$ is an increasing sequence of measure zero and closed sets of reals in $[0,1]$, then $ \bigcup_{n \in\N } X_n$ has measure zero.
\end{princ} 
By the the main result of \cite{trohim}, not all nowhere dense measure zero sets are the countable union of measure zero closed sets, i.e.\ $\PHP_{[0,1]}$ does not generate `all' measure zero sets. 

\smallskip

Secondly, fragments of the induction axiom are sometimes used in an essential way in second-order RM (see e.g.\ \cite{neeman}).  
An important role of induction is to provide `finite comprehension' (see \cite{simpson2}*{X.4.4}).
As in \cite{samBIG2}, we need the following fragment of finite comprehension, provable from the induction axiom. 
\begin{princ}[$\IND_{\R}$]
For $F:(\R\times \N)\di \N, k\in \N$, there is $X\subset \N$ such that
\[
(\forall n\leq k)\big[ (\exists x\in \R)(F(x, n)=0)\asa n\in X\big].
\]
\end{princ}
\noindent
In particular, the following rather important result seems to require $\IND_{\R}$.  
\begin{thm}[$\ACAo+\IND_{\R}$]\label{zonk}
For Riemann integrable $f:[0,1]\di \R$ with oscillation function $\osc_{f}:[0,1]\di \R$, 
the set $D_{k}:=\{x\in [0,1]:\osc_{f}(x)\geq \frac{1}{2^{k}}\}$ is measure zero for any fixed $k\in \N$.
\end{thm}
\begin{proof}
The theorem follows from \cite{samBIG2}*{Cor.\ 3.4}.  A sketch of the proof of the latter is as follows: let $f:[0,1]\di \R$ be Riemann integrable with oscillation function $\osc_{f}:[0,1]\di \R$, 
such that $D_{k_{0}}:=\{x\in [0,1]:\osc_{f}(x)\geq \frac{1}{2^{k_{0}}}\}$ has measure $\eps>0$ for some fixed $k_{0}\in \N$.  
For a partition $P$ given as $0=x_{0}, t_{0},x_{1}, t_{1}, x_{1}, \dots, x_{k-1}, t_{k}, x_{k}=1 $ with small enough mesh, one obtains two partitions $P', P''$: as follows: in case $[x_{i}, x_{i+1}]\cap D_{k_{0}}\ne\emptyset$, replace $t_{i}$ by respectively $t_{i}'$ and $t_{i}''$ such that $|f(t_{i}')-f(t_{i}'')|\geq \frac{1}{2^{k_{0}}}$; these reals exist by definition of $D_{k_{0}}$.  
By definition, $S(f, P')$ and $S(f, P'')$ are at least $\eps/2^{k_{0}+1}$ apart, i.e.\ $f$ is not Riemann integrable, a contradiction.  The definition of $P', P''$ can be formalised using $\IND_{\R}$. 
\end{proof}
Thirdly, we establish Theorem \ref{duck555} where we recall the set $B_{f}$ from \eqref{BF}. 
\begin{thm}[$\ACAo+\IND_{\R}+\QFAC^{0,1}$]\label{duck555}
The following are equivalent.  
\begin{enumerate}
\renewcommand{\theenumi}{\alph{enumi}}
\item The pigeonhole principle for measure spaces as in $\PHP_{[0,1]}$.
\item \(Vitali-Lebesgue\) For Riemann integrable $f:[0,1]\di \R$ with an oscillation function, the set $ C_{f}$ has measure $1$.\label{tao0}
\item For Riemann integrable $f:[0,1]\di \R$ with an oscillation function, the set $ B_{f}$ has measure $1$.\label{tao00}
\item For Riemann integrable $f:[0,1]\di \R$ with an oscillation function, the set of all $x\in (0,1)$ such that \eqref{zolk} \(or \eqref{fereng}\) has measure $1$.\label{tao000}
\item For Riemann integrable and pointwise discontinuous $f:[0,1]\di \R$ with an oscillation function, the set $ B_{f}$ has measure $1$.\label{taowa}
\item For Riemann integrable lsco $f:[0,1]\di \R$, the set $ B_{f}$ has measure $1$.\label{tao2}
\end{enumerate}
We can replace `usco' by `cliquish with an oscillation function' in the above.  
\end{thm}
\begin{proof}
First of all, assume $\PHP_{[0,1]}$ and let $f:[0,1]\di \R$ be as in item \eqref{tao0} of the theorem.
By Theorem \ref{zonk}, each $D_{k}:=\{x\in [0,1]:\osc_{f}(x)\geq \frac{1}{2^{k}}\}$ has measure zero.
By \cite{samBIG2}*{Theorem 1.17}, $\osc_{f}$ is usco and hence $D_{k}$ is closed; both facts essentially follow by definition as well.  
Then $\PHP_{[0,1]}$ implies that $D_{f}=\cup_{k}D_{k}$ has measure zero, yielding item \eqref{tao0}.
By Theorem \ref{echtdecrux}, $B_{f}$ has measure one, i.e.\ item \eqref{tao00} follows, and the same for items \eqref{tao000}-\eqref{taowa}.  

\smallskip

For item \eqref{tao2}, we proceed in essentially the same way: $D_{f}$ exists in $\ACAo+\QFAC^{0,1}$ by \cite{samBIG2}*{Theorem 2.4}, and we just need to define some kind of `replacement' set for $D_{k}$. 
To this end, let $f:[0,1]\di \R$ be lsco and consider:
\be\label{henz}\textstyle
 f(x)\leq q \wedge (\forall N\in \N)(\exists z\in B(x, \frac{1}{2^{N}})~\cap~\Q)( f(z)>q+\frac{1}{2^{l}}  ).
\ee
Using $(\exists^{2})$, let $E_{q, l}$ be the set of all $x\in [0,1]$ satisfying \eqref{henz}.  
Since $f$ is lsco, \eqref{henz} is -essentially by definition- equivalent to:
\begin{center}
 \emph{the formula \eqref{henz} with `$~\cap~\Q$' omitted}.
 \end{center}
Intuitively, the set $E_{q,l}$ provides a `replacement' set for $D_{k}$. 
Indeed, since $f$ is lsco, the set $E_{q,l}$ is closed.  Moreover, we also have $D_{f}=\cup_{l\in \N, q\in \Q}E_{q,l}$, by the epsilon-delta definition of (local) continuity. 
That $E_{q, l}$ has measure zero is proved in the same way as for $D_{k}$ in Theorem \ref{zonk}. 
Again, $\PHP_{[0,1]}$ implies that $D_{f}$ has measure zero and Theorem \ref{echtdecrux} implies that $B_{f}$ has measure one, as required for item \eqref{tao2}.  
 
\smallskip

To derive $\PHP_{[0,1]}$ from item \eqref{tao2} (or the items \eqref{tao0}-\eqref{taowa}), let $(X_{n})_{n\in \N}$ be an increasing sequence of closed and measure zero sets.  
Since $\Q$ has an enumeration, if $\cup_{n\in \N}X_{n}\setminus \Q$ has measure zero, so does $\cup_{n\in \N}X_{n}$, say in $\ACAo$.  
Hence, we may assume that $\Q\cap \cup_{n\in \N}X_{n}=\emptyset$. 
Now consider the following function, which is usco, cliquish, and its own oscillation function by \cite{samBIG2}*{Theorems 1.16-18}:
\be\label{mopi}
h(x):=
\begin{cases}
0 & x\not \in \cup_{m\in \N}X_{m} \\
\frac{1}{2^{n+1}} &  x\in X_{n} \textup{ and $n$ is the least such number}
\end{cases},
\ee
and which satisfies $B_{n}(h, x)=0$ for all $x\in [0,1]$.
Now, $h$ is Riemann integrable with integral equal to zero on $[0,1]$, which can be proved using the obvious\footnote{Fix $\eps>0$ and $n_{0}\in \N$ such that $\frac{1}{2^{n_{0}}}<\eps$.  Then $\ACAo+\QFAC^{0,1}$ proves that for any partition $P$ with mesh $<\frac{1}{2^{n_{0}+2}}$, the Riemann sum $S(h, P)$ is $<\eps$ in absolute value.  To this end, cover $X_{n_{0}}$ by a sequence of intervals of total length at most $\frac{1}{2^{n_{0}+2}}$ and find a finite sub-covering.} epsilon-delta proof.  Moreover, since $\Q\cap \cup_{n\in \N}X_{n}=\emptyset$, $h$ is continuous at any rational and therefore pointwise discontinuous, i.e.\ the equivalence for item \eqref{taowa} readily follows.  
By item \eqref{tao2}, $B_{f}$ has measure $1$, implying that for almost all $x\in (0,1)$, we have $h(x)=\lim_{n\di \infty} B_{n}(h,x )=0$, i.e.\ $\cup_{n\in \N}X_{n}$ has measure zero, and $\PHP_{[0,1]}$ follows.
Note that item \eqref{tao000} also guarantees that $h(x)=0$ almost everywhere, i.e.\ $\cup_{n\in \N}X_{n}$ has measure zero.
For the final sentence of the theorem, recall that the function $h:[0,1]\di \R$ from \eqref{mopi} is cliquish.  
\end{proof}
In light of Theorem \ref{zonk}, Riemann integrable functions can \emph{almost} be proved to be continuous ae.  
However, $\PHP_{[0,1]}$ is needed to establish that $D_{f}=\cup_{k\in \N}D_{k}$ has measure zero.  
The restriction in item \eqref{taowa} is non-trivial as is it consistent with $\Z_{2}^{\omega}$ that there are Riemann integrable functions that are \emph{totally discontinuous}.  

\smallskip

As to generalisations of the previous theorem, there are a surprisingly large number of rather diverse equivalent definitions of `Baire 1' on the reals (\cites{leebaire,beren,koumer}), including \emph{$B$-class-1-measurability} and \emph{fragmentedness} by \cite{koumer}*{Theorem~2.3} and \cite{kura}*{\S34, VII}.   One readily shows that $h$ from \eqref{mopi} satisfies these definitions, say in $\ACAo$, i.e.\ the previous theorem can be formulated for these notions.  By contrast, the restriction of item \eqref{tao0} or \eqref{taowa} in Theorem \ref{duck555} to (the standard definition of) Baire 1 functions, can be proved in $\ACAo$ by \cite{samBIG2}*{Theorem 3.7}.  In general, the function $h$ from \eqref{mopi} is rather well-behaved and therefore included in many (lesser known than Baire 1) function classes.

\subsection{Baire category theorem}\label{BKT}
We introduce the Baire category theorem and obtain equivalences involving the approximation theorem for usco and cliquish functions via Bernstein polynomials.  
The RM of $\BCT_{[0,1]}$ as in Principle \ref{BCTI} and $\PHP_{[0,1]}$ is often similar (see \cite{samBIG2}), which is why we treat the former in less detail.  

\smallskip

First of all, we shall study the Baire category theorem formulated as follows.  
We have established a substantial number of equivalences in \cite{samBIG2} between this principle and basic properties of usco functions and related classes.
\begin{princ}[$\BCT_{[0,1]}$]\label{BCTI}
If $ (O_n)_{n \in \N}$ is a decreasing sequence of dense open sets of reals in $[0,1]$, then $ \bigcap_{n \in\N } O_n$ is non-empty.
\end{princ} 
\noindent
We assume that $O_{n+1}\subseteq O_{n}$ for all $n\in \N$ to avoid the use of induction to prove that a finite intersection of open and dense sets is again open and dense. 

\smallskip

Secondly, we have the following theorem where we recall the set $B_{f}$ from \eqref{BF}.  
\begin{thm}[$\ACAo$]\label{Y}
The following are equivalent.  
\begin{enumerate}
\renewcommand{\theenumi}{\alph{enumi}}
\item The Baire category theorem as in $\BCT_{[0,1]}$.\label{bao0}
\item For usco $f:[0,1]\di \R$, the set $ C_{f}$ is non-empty.\label{bao01}
\item For cliquish $f:[0,1]\di \R$ with an oscillation function, we have $ C_{f}\ne\emptyset$.\label{bao001}
\item For usco $f:[0,1]\di \R$, the set $ B_{f}$ is non-empty.\label{bao1}
\item For cliquish $f:[0,1]\di \R$ with an oscillation function, $ B_{f}$ is non-empty.\label{bao2}
\end{enumerate}
\end{thm}
\begin{proof}
For the implications \eqref{bao1}$\di$\eqref{bao0} and \eqref{bao2}$\di\eqref{bao0}$, let $ (O_n)_{n \in \N}$ be a decreasing sequence of dense open sets of reals in $[0,1]$.  
In case there is $q\in \Q $ with $q\in \cap_{n\in \N}O_{n}$, $\BCT_{[0,1]}$ follows.  For the case where $\emptyset =\Q\cap \left(\cap_{n\in \N}O_{n}\right)$, we proceed as follows. 
For $X_{n}:= [0,1]\setminus O_{n}$, consider $h$ as in \eqref{mopi}, which is usco, cliquish, and its own oscillation function by \cite{samBIG2}*{Theorems~1.16-18}.  Hence, the set $B_{h}$ from \eqref{BF} is non-empty given item~\eqref{bao1} or \eqref{bao2}.   We now show that $\lim_{n\di \infty}B_{n}(h, x)=0$ for all $x\in [0,1]$, finishing the proof.  To this end, recall that $\Q\subset \cup_{n\in \N}X_{n}$ and define $Q_{n}:= X_{n}\cap \Q $ where $\Q=\cup_{n\in \N}Q_{n}$.   
Now consider $\tilde{h}:[0,1]\di \R$ as follows  
\be\label{mopit}
\tilde{h}(x):=
\begin{cases}
0 & x\not \in \Q \\
\frac{1}{2^{n+1}} &  x\in Q_{n} \textup{ and $n$ is the least such number}
\end{cases},
\ee
which is really a modification of Thomae's function (see \cites{thomeke,dagsamXIV}).
Then $\tilde{h}$ is continuous on $[0,1]\setminus \Q$ since each $X_{n}$ is nowhere dense and closed.  
By Theorem \ref{echtdecrux}, we have $0=\tilde{h}(x)=\lim_{n\di\infty}B_{n}(\tilde{h}, x)$ for $x\in (0,1)\setminus \Q$.  
By the uniform continuity of the polynomials $p_{k, n}(x)$ on $[0,1]$, we also have $\lim_{n\di\infty}B_{n}(\tilde{h}, q)=0$ for $q\in [0,1]\cap \Q$.
Since Bernstein polynomials only invoke rational function values, we have $B_{n}(h, x)=B_{n}(\tilde{h}, x)$ for any $x\in [0,1]$ and $n\in \N$.
Hence, we conclude  $\lim_{n\di\infty}B_{n}({h}, x)=0$ for $x\in [0,1]$, as required. 

\smallskip

Finally, the implications \eqref{bao0}$\di$\eqref{bao01}$\di$\eqref{bao1} and \eqref{bao0}$\di$\eqref{bao001}$\di$\eqref{bao2} follow from \cite{samBIG2}*{Theorem 2.3 and 2.16} and Theorem \ref{echtdecrux}. 
For the first implication, one now defines a variation of $E_{q, l}$ based on \eqref{henz} to replace $D_{k}$, while the third one follows by the textbook proof.  
\end{proof}
As also stated in \cite{samBIG2}, much to our own surprise, the `counterexample' function from \eqref{mopi} has nice properties \emph{that are provable in weak systems}, including the behaviour of the associated Bernstein polynomials.    

\smallskip

Finally, we discuss certain restrictions of the above theorems. 
\begin{rem}[Restrictions, trivial and otherwise]\label{WML}\rm
It should not be a surprise that suitable restrictions of principles that imply $\NIN_{[0,1]}$, are again provable from the (second-order) Big Five.  
We discuss three examples pertaining to the above where two are perhaps surprising.  

\smallskip

First of all, item \eqref{LB} in Theorem \ref{tach} is \emph{equivalent} to $\ATR_{0}$ if we restrict to functions with an arithmetical or $\Sigma_{1}^{1}$-graph by \cite{dagsamXIV}*{\S2.6}. 
The results in the latter pertain to bounded variation functions but are readily adapted to regulated functions, which also only have countably many points of discontinuity. 
Thus, the same restriction of item \eqref{LC} in Theorem \ref{tach} is {equivalent} to $\ATR_{0}$ in the same way. 
Similar results hold for the restriction to quasi-continuous functions.  

\smallskip

Secondly, item \eqref{LB} in Theorem \ref{tach} is provable\footnote{Here, $\ATR_{0}$ is just an upper bound and better results can be found in \cite{samBIG}*{\S3.5}.} from $\ATR_{0}$ (and extra induction) when restricted to Baire 1 functions by \cite{dagsamXIV}*{\S2.6}.
By Theorem \ref{echtdecrux}, the same holds for the restriction to Baire 1 functions of \eqref{LC} in Theorem \ref{tach}.  However, this restriction does not seem to be a `real' one as many a textbook tells us that:
\be\label{oip2}
\textup{\emph{regulated functions {are}  Baire 1 on the reals}.}
\ee
Of course, \eqref{oip2} is true \emph{but} $\Z_{2}^{\omega}+\QFAC^{0,1}$ cannot prove it by \cite{dagsamXIV}*{Theorem~2.35}.  
In general, the usual picture of the classical hierarchy of function classes looks \emph{very} different in weak logical systems and even in $\Z_{2}^{\omega}$.  
Indeed, by Theorem~\ref{tomma}, it is consistent with $\Z_{2}^{\omega}+\QFAC^{0,1}$ that there are regulated functions that are \emph{discontinuous everywhere}.  

\smallskip

Thirdly, similar to the previous paragraph, item \eqref{bao01} of Theorem \ref{Y} restricted to Baire 1 functions is provable in $\ACAo$ by \cite{samBIG2}*{Theorem 2.9}.  
Again, this restriction does not seem `real' as it is well-known that:
\be\label{oip}
\textup{\emph{usco functions {are}  Baire 1 on the reals},} 
\ee
but $\Z_{2}^{\omega}+\QFAC^{0,1}$ cannot prove \eqref{oip} by \cite{samBIG2}*{Cor~2.8}.  Unfortunately, the second-order coding of usco and lsco functions from \cite{ekelhaft} is such that 
the Baire 1 representation of usco and lsco is `baked into' the coding, in light of \cite{ekelhaft}*{\S6}.  
It would therefore be more correct to refer to the representations from \cite{ekelhaft} as \emph{codes for usco functions that are also Baire 1}.
What is worse, \eqref{oip} implies $\enum$ is therefore not an `innocent' background assumption in RM, as $\ACAo+\enum$ proves $\ATR_{0}$ and $\FIVE^{\omega}+\enum$ proves $\SIX$ (see Section \ref{senum}).

\end{rem}

\begin{ack}\rm 
We thank Anil Nerode for his valuable advice, especially the suggestion of studying nsc for the Riemann integral, and discussions related to Baire classes.
We thank Dave L.\ Renfro for his efforts in providing a most encyclopedic summary of analysis, to be found online.  
Our research was supported by the \emph{Klaus Tschira Boost Fund} via the grant Projekt KT43.
%
We express our gratitude towards the latter institution.    
\end{ack}
%


\appendix
\section{Technical Appendix: introducing Reverse Mathematics}\label{RMA}
We discuss the language of Reverse Mathematics (Section \ref{introrm}) and introduce -in full detail- Kohlenbach's base theory of \emph{higher-order} Reverse Mathematics (Section~\ref{rmbt}).
Some common notations may be found in Section \ref{kkk}.
\subsection{Introduction}\label{introrm}
We sketch some aspects of Kohlenbach's \emph{higher-order} RM (\cite{kohlenbach2}) essential to this paper, including the base theory $\RCAo$ (Definition \ref{kase}).  

\smallskip

First of all, in contrast to `classical' RM based on \emph{second-order arithmetic} $\Z_{2}$, higher-order RM uses $\L_{\omega}$, the richer language of \emph{higher-order arithmetic}.  
Indeed, while the former is restricted to natural numbers and sets of natural numbers, higher-order arithmetic can accommodate sets of sets of natural numbers, sets of sets of sets of natural numbers, et cetera.  
To formalise this idea, we introduce the collection of \emph{all finite types} $\mathbf{T}$, defined by the two clauses:
\begin{center}
(i) $0\in \mathbf{T}$   and   (ii)  If $\sigma, \tau\in \mathbf{T}$ then $( \sigma \di \tau) \in \mathbf{T}$,
\end{center}
where $0$ is the type of natural numbers, and $\sigma\di \tau$ is the type of mappings from objects of type $\sigma$ to objects of type $\tau$.
In this way, $1\equiv 0\di 0$ is the type of functions from numbers to numbers, and  $n+1\equiv n\di 0$.  Viewing sets as given by characteristic functions, we note that $\Z_{2}$ only includes objects of type $0$ and $1$.    

\smallskip

Secondly, the language $\L_{\omega}$ includes variables $x^{\rho}, y^{\rho}, z^{\rho},\dots$ of any finite type $\rho\in \mathbf{T}$.  Types may be omitted when they can be inferred from context.  
The constants of $\L_{\omega}$ include the type $0$ objects $0, 1$ and $ <_{0}, +_{0}, \times_{0},=_{0}$  which are intended to have their usual meaning as operations on $\N$.
Equality at higher types is defined in terms of `$=_{0}$' as follows: for any objects $x^{\tau}, y^{\tau}$, we have
\be\label{aparth}
[x=_{\tau}y] \equiv (\forall z_{1}^{\tau_{1}}\dots z_{k}^{\tau_{k}})[xz_{1}\dots z_{k}=_{0}yz_{1}\dots z_{k}],
\ee
if the type $\tau$ is composed as $\tau\equiv(\tau_{1}\di \dots\di \tau_{k}\di 0)$.  
Furthermore, $\L_{\omega}$ also includes the \emph{recursor constant} $\mathbf{R}_{\sigma}$ for any $\sigma\in \mathbf{T}$, which allows for iteration on type $\sigma$-objects as in the special case \eqref{special}.  Formulas and terms are defined as usual.  
One obtains the sub-language $\L_{n+2}$ by restricting the above type formation rule to produce only type $n+1$ objects (and related types of similar complexity).        

\subsection{The base theory of higher-order Reverse Mathematics}\label{rmbt}
We introduce Kohlenbach's base theory $\RCAo$, first introduced in \cite{kohlenbach2}*{\S2}.
\bdefi\label{kase} 
The base theory $\RCAo$ consists of the following axioms.
\begin{enumerate}
 \renewcommand{\theenumi}{\alph{enumi}}
\item  Basic axioms expressing that $0, 1, <_{0}, +_{0}, \times_{0}$ form an ordered semi-ring with equality $=_{0}$.
\item Basic axioms defining the well-known $\Pi$ and $\Sigma$ combinators (aka $K$ and $S$ in \cite{avi2}), which allow for the definition of \emph{$\lambda$-abstraction}. 
\item The defining axiom of the recursor constant $\mathbf{R}_{0}$: for $m^{0}$ and $f^{1}$: 
\be\label{special}
\mathbf{R}_{0}(f, m, 0):= m \textup{ and } \mathbf{R}_{0}(f, m, n+1):= f(n, \mathbf{R}_{0}(f, m, n)).
\ee
\item The \emph{axiom of extensionality}: for all $\rho, \tau\in \mathbf{T}$, we have:
\be\label{EXT}\tag{$\textsf{\textup{E}}_{\rho, \tau}$}  
(\forall  x^{\rho},y^{\rho}, \varphi^{\rho\di \tau}) \big[x=_{\rho} y \di \varphi(x)=_{\tau}\varphi(y)   \big].
\ee 
\item The induction axiom for quantifier-free formulas of $\L_{\omega}$.
\item $\QFAC^{1,0}$: the quantifier-free Axiom of Choice as in Definition \ref{QFAC}.
\end{enumerate}
\edefi
\noindent
Note that variables (of any finite type) are allowed in quantifier-free formulas of the language $\L_{\omega}$: only quantifiers are banned.
Recursion as in \eqref{special} is called \emph{primitive recursion}; the class of functionals obtained from $\mathbf{R}_{\rho}$ for all $\rho \in \mathbf{T}$ is called \emph{G\"odel's system $T$} of all (higher-order) primitive recursive functionals. 
\bdefi\label{QFAC} The axiom $\QFAC$ consists of the following for all $\sigma, \tau \in \textbf{T}$:
\be\tag{$\QFAC^{\sigma,\tau}$}
(\forall x^{\sigma})(\exists y^{\tau})A(x, y)\di (\exists Y^{\sigma\di \tau})(\forall x^{\sigma})A(x, Y(x)),
\ee
for any quantifier-free formula $A$ in the language of $\L_{\omega}$.
\edefi
As discussed in \cite{kohlenbach2}*{\S2}, $\RCAo$ and $\RCA_{0}$ prove the same sentences `up to language' as the latter is set-based and the former function-based.   
This conservation results is obtained via the so-called $\ECF$-interpretation, which we now discuss. 
\begin{rem}[The $\ECF$-interpretation]\rm
The (rather) technical definition of $\ECF$ may be found in \cite{troelstra1}*{p.\ 138, \S2.6}.
Intuitively, the $\ECF$-interpretation $[A]_{\ECF}$ of a formula $A\in \L_{\omega}$ is just $A$ with all variables 
of type two and higher replaced by type one variables ranging over so-called `associates' or `RM-codes' (see \cite{kohlenbach4}*{\S4}); the latter are (countable) representations of continuous functionals.  
The $\ECF$-interpretation connects $\RCAo$ and $\RCA_{0}$ (see \cite{kohlenbach2}*{Prop.\ 3.1}) in that if $\RCAo$ proves $A$, then $\RCA_{0}$ proves $[A]_{\ECF}$, again `up to language', as $\RCA_{0}$ is 
formulated using sets, and $[A]_{\ECF}$ is formulated using types, i.e.\ using type zero and one objects.  
\end{rem}
In light of the widespread use of codes in RM and the common practise of identifying codes with the objects being coded, it is no exaggeration to refer to $\ECF$ as the \emph{canonical} embedding of higher-order into second-order arithmetic.

\subsection{Notations and the like}\label{kkk}
We introduce the usual notations for common mathematical notions, like real numbers, as also introduced in \cite{kohlenbach2}.  
\begin{defi}[Real numbers and related notions in $\RCAo$]\label{keepintireal}\rm~
\begin{enumerate}
 \renewcommand{\theenumi}{\alph{enumi}}
\item Natural numbers correspond to type zero objects, and we use `$n^{0}$' and `$n\in \N$' interchangeably.  Rational numbers are defined as signed quotients of natural numbers, and `$q\in \Q$' and `$<_{\Q}$' have their usual meaning.    
\item Real numbers are coded by fast-converging Cauchy sequences $q_{(\cdot)}:\N\di \Q$, i.e.\  such that $(\forall n^{0}, i^{0})(|q_{n}-q_{n+i}|<_{\Q} \frac{1}{2^{n}})$.  
We use Kohlenbach's `hat function' from \cite{kohlenbach2}*{p.\ 289} to guarantee that every $q^{1}$ defines a real number.  
\item We write `$x\in \R$' to express that $x^{1}:=(q^{1}_{(\cdot)})$ represents a real as in the previous item and write $[x](k):=q_{k}$ for the $k$-th approximation of $x$.    
\item Two reals $x, y$ represented by $q_{(\cdot)}$ and $r_{(\cdot)}$ are \emph{equal}, denoted $x=_{\R}y$, if $(\forall n^{0})(|q_{n}-r_{n}|\leq {2^{-n+1}})$. Inequality `$<_{\R}$' is defined similarly.  
We sometimes omit the subscript `$\R$' if it is clear from context.           
\item Functions $F:\R\di \R$ are represented by $\Phi^{1\di 1}$ mapping equal reals to equal reals, i.e.\ extensionality as in $(\forall x , y\in \R)(x=_{\R}y\di \Phi(x)=_{\R}\Phi(y))$.\label{EXTEN}
\item The relation `$x\leq_{\tau}y$' is defined as in \eqref{aparth} but with `$\leq_{0}$' instead of `$=_{0}$'.  Binary sequences are denoted `$f^{1}, g^{1}\leq_{1}1$', but also `$f,g\in C$' or `$f, g\in 2^{\N}$'.  Elements of Baire space are given by $f^{1}, g^{1}$, but also denoted `$f, g\in \N^{\N}$'.
\item For a binary sequence $f^{1}$, the associated real in $[0,1]$ is $\r(f):=\sum_{n=0}^{\infty}\frac{f(n)}{2^{n+1}}$.\label{detrippe}
\item Sets of type $\rho$ objects $X^{\rho\di 0}, Y^{\rho\di 0}, \dots$ are given by their characteristic functions $F^{\rho\di 0}_{X}\leq_{\rho\di 0}1$, i.e.\ we write `$x\in X$' for $ F_{X}(x)=_{0}1$. \label{koer} 
\end{enumerate}
\end{defi}
For completeness, we list the following notational convention for finite sequences.  
\begin{nota}[Finite sequences]\label{skim}\rm
The type for `finite sequences of objects of type $\rho$' is denoted $\rho^{*}$, which we shall only use for $\rho=0,1$.  
Since the usual coding of pairs of numbers goes through in $\RCAo$, we shall not always distinguish between $0$ and $0^{*}$. 
Similarly, we assume a fixed coding for finite sequences of type $1$ and shall make use of the type `$1^{*}$'.  
In general, we do not always distinguish between `$s^{\rho}$' and `$\langle s^{\rho}\rangle$', where the former is `the object $s$ of type $\rho$', and the latter is `the sequence of type $\rho^{*}$ with only element $s^{\rho}$'.  The empty sequence for the type $\rho^{*}$ is denoted by `$\langle \rangle_{\rho}$', usually with the typing omitted.  

\smallskip

Furthermore, we denote by `$|s|=n$' the length of the finite sequence $s^{\rho^{*}}=\langle s_{0}^{\rho},s_{1}^{\rho},\dots,s_{n-1}^{\rho}\rangle$, where $|\langle\rangle|=0$, i.e.\ the empty sequence has length zero.
For sequences $s^{\rho^{*}}, t^{\rho^{*}}$, we denote by `$s*t$' the concatenation of $s$ and $t$, i.e.\ $(s*t)(i)=s(i)$ for $i<|s|$ and $(s*t)(j)=t(|s|-j)$ for $|s|\leq j< |s|+|t|$. For a sequence $s^{\rho^{*}}$, we define $\overline{s}N:=\langle s(0), s(1), \dots,  s(N-1)\rangle $ for $N^{0}<|s|$.  
For a sequence $\alpha^{0\di \rho}$, we also write $\overline{\alpha}N=\langle \alpha(0), \alpha(1),\dots, \alpha(N-1)\rangle$ for \emph{any} $N^{0}$.  By way of shorthand, 
$(\forall q^{\rho}\in Q^{\rho^{*}})A(q)$ abbreviates $(\forall i^{0}<|Q|)A(Q(i))$, which is (equivalent to) quantifier-free if $A$ is.   
\end{nota}

\bye